\documentclass[a4paper,12pt,reqno]{amsart}
\usepackage{amsmath,amssymb,ifthen}
\usepackage{fullpage}
\usepackage{graphicx,psfrag}
\usepackage{color}
\usepackage{moreverb}
\usepackage{amsthm}
\usepackage{mathrsfs}
\usepackage{hhline}
\usepackage{bbm}
\usepackage[english]{babel}
\usepackage{tikz}
\usetikzlibrary{calc,math}
\usepackage{pgfplots}
\pgfplotsset{compat=newest}
\usepackage[utf8]{inputenc}
\usepackage[T1]{fontenc}
\usepackage{amsfonts}
\usepackage{enumerate}
\usepackage{caption}
\usepackage{subcaption}
\usepackage{xcolor}
\usepackage{hyperref}

\newcommand\coarse{\bullet}
\newcommand\fine{\circ}

\renewcommand{\d}{\,{\rm d}} 
\DeclareMathOperator{\diam}{diam}

\DeclareMathOperator\osc{osc}

\newcommand{\const}[1]{C_{\text{\rm#1}}}

\newcommand{\set}[2]{\big\{#1\,:\,#2\big\}}

\newcommand{\norm}[3][]{#1\|#2#1\|_{#3}}

\newcommand\refine{{\tt refine}}

\newcommand\N{\mathbb{N}}
\newcommand\R{\mathbb{R}}
\newcommand\T{\mathbb{T}}

\newcommand\MM{\mathcal M}
\newcommand\OO{{\mathcal O}}
\newcommand\TT{\mathcal T}

\numberwithin{equation}{section}
\numberwithin{figure}{section}
\newtheorem{theorem}{Theorem}[section]
\newtheorem{proposition}[theorem]{Proposition}

\newtheorem{algorithm}[theorem]{Algorithm}

\newtheorem{remark}[theorem]{Remark}

\newcommand*\patchAmsMathEnvironmentForLineno[1]{%
  \expandafter\let\csname old#1\expandafter\endcsname\csname #1\endcsname
  \expandafter\let\csname oldend#1\expandafter\endcsname\csname end#1\endcsname
  \renewenvironment{#1}%
     {\linenomath\csname old#1\endcsname}%
     {\csname oldend#1\endcsname\endlinenomath}}%
\newcommand*\patchBothAmsMathEnvironmentsForLineno[1]{%
  \patchAmsMathEnvironmentForLineno{#1}%
  \patchAmsMathEnvironmentForLineno{#1*}}%
\AtBeginDocument{%
\patchBothAmsMathEnvironmentsForLineno{equation}%
\patchBothAmsMathEnvironmentsForLineno{align}%
\patchBothAmsMathEnvironmentsForLineno{flalign}%
\patchBothAmsMathEnvironmentsForLineno{alignat}%
\patchBothAmsMathEnvironmentsForLineno{gather}%
\patchBothAmsMathEnvironmentsForLineno{multline}%
}
\usepackage[mathlines]{lineno}

\makeatletter
\def\@seccntformat#1{%
  \protect\textup{\protect\@secnumfont
    \ifnum\pdfstrcmp{subsection}{#1}=0 \bfseries\fi
    \csname the#1\endcsname
    \protect\@secnumpunct
  }%
}  
\makeatother

\newcommand{\Hloc}{H_{\mathrm{loc}}}
\newcommand{\inner}[3][]{\langle #2,#3 \rangle_{#1}}

\newcommand{\n}{\boldsymbol{n}}
\newcommand{\x}{\boldsymbol{x}}
\newcommand{\y}{\boldsymbol{y}}
\newcommand{\z}{\boldsymbol{z}}
\renewcommand{\u}{\boldsymbol{u}}
\renewcommand{\v}{\boldsymbol{v}}

\newcommand{\dy}{\mathrm{d}\y}
\newcommand{\Omegaext}{\Omega^{\mathrm{ext}}}
\newcommand{\PP}{\mathcal{P}}
\newcommand{\RR}{\mathcal{R}}
\renewcommand{\SS}{\mathcal{S}}

\newcommand{\uext}{u^{\mathrm{ext}}}
\newcommand{\abs}[1]{\lvert #1 \rvert}

\begin{document}
\title{Optimal convergence rates\\ of an adaptive hybrid FEM-BEM method\\ for full-space linear transmission problems}
\author{Gregor~Gantner}
\address{Institute for Numerical Simulation, University of Bonn \\
Friedrich-Hirzebruch-Allee~7, 53115 Bonn, Germany
}
\email{gantner@ins-uni.bonn.de}

\author{Michele~Ruggeri}
\address{Department of Mathematics,
University of Bologna,
Piazza di Porta San Donato 5, 40126 Bologna, Italy}
\email{m.ruggeri@unibo.it}
\dedicatory{Dedicated to Dirk Praetorius on the occasion of his 50th birthday.}
\thanks{\emph{Acknowledgements.}
GG and MR acknowledge that this work was initiated during their employment at Inria Paris and University of Strathclyde, respectively.
These institutions are thankfully acknowledged.
In addition, GG acknowledges funding by the Deutsche Forschungsgemeinschaft (DFG, German Research Foundation) under Germany's Excellence Strategy – EXC-2047/1 – 390685813.
MR is a member of the `Gruppo Nazionale per il Calcolo Scientifico (GNCS)'
of the Italian `Istituto Nazionale di Alta Matematica (INdAM)'
and was partially supported by GNCS (research project GNCS 2024 on \emph{Advanced numerical methods for nonlinear problems in materials science} -- CUP E53C23001670001)
and by the European Union -- NextGenerationEU under the National Recovery and Resilience Plan (PNRR) -- Mission 4 Education and research -- Component 2 From research to business -- Investment 1.1 Notice Prin 2022 -- DD N. 104 of 2/2/2022, entitled \emph{Low-rank Structures and Numerical Methods in Matrix and Tensor Computations and their Application}, code 20227PCCKZ -- CUP J53D23003620006.}
\begin{abstract}
We consider a hybrid FEM-BEM method 
to compute approximations of full-space linear elliptic transmission problems.
First, we derive \textsl{a~priori} and \textsl{a~posteriori} error estimates. 
Then, building on the latter, we present an adaptive algorithm
and prove that it converges at optimal rates with respect to the number of mesh elements. 
Finally, we provide numerical experiments, demonstrating the practical performance of the adaptive algorithm.
\end{abstract}
\maketitle

\section{Introduction}
Let $\Omega \subset \R^d$ ($d=2,3$) be a bounded Lipschitz domain with polytopal boundary $\Gamma := \partial\Omega$
and outward-pointing unit normal vector $\n: \Gamma \to \R^d$.
Given $f: \Omega \to \R$ and $g, \phi : \Gamma \to \R$,
we are interested in computing a numerical approximation of the solution pair $u: \Omega \to \R$, $\uext: \Omegaext:=\R^d\setminus\overline\Omega \to \R$ 
of the full-space linear elliptic transmission problem
\begin{subequations} \label{eq:transmission}
\begin{alignat}{2}
\label{eq:poisson_interior}
- \Delta u &= f &\quad& \text{in } \Omega,\\
- \Delta \uext &= 0 && \text{in } \Omegaext,\\
\label{eq:jump_cond1}
u - \uext &= g && \text{on } \Gamma,\\
\label{eq:jump_cond2}
\partial_{\n}(u - \uext) &= \phi && \text{on } \Gamma,\\
\uext(\x) &= c\log\abs{\x} + \OO(\abs{\x}^{-1}) && \text{as } \abs{\x} \to \infty,
\end{alignat}
\end{subequations}
for some arbitrary $c\in\R$ if $d=2$ and $c=0$ if $d=3$.

To cope with the unboundedness of the exterior domain $\Omegaext$,
well-known approaches usually referred to as \emph{FEM-BEM coupling methods}
resort to combinations of finite element approximations of the problem in the interior domain
with certain reformulations of the problem in the exterior domain in terms of boundary integral operators;
see, e.g., the seminal papers~\cite{jn1980,bm1983,costabel1988}.
Such methods usually require the solution of linear systems in which the system matrix 
involves blocks of sparse matrices (arising from finite element approximations in the interior domain) 
and fully-populated matrices (resulting from the discretization of nonlocal boundary integral operators).

In this work, we consider a hybrid FEM-BEM method for full-space linear elliptic transmission problems, 
which was originally proposed for micromagnetic applications~\cite{fk1990}. 
More precisely, the approach was designed to approximate solutions of the special case of~\eqref{eq:transmission}
in which $f = - \nabla\cdot\boldsymbol{m}$, $g=0$, and $\phi = \boldsymbol{m} \cdot \n$
for a given vector-valued function $\boldsymbol{m}$ (representing the magnetization of a ferromagnetic material).
In this case, the solution $u$ of~\eqref{eq:transmission} models the magnetostatic potential 
and the resulting magnetostatic field $\boldsymbol{h} = - \nabla u$ is an interaction that needs to be computed
to simulate the magnetization dynamics; see, e.g., \cite{bffgpprs2014,arbvhps15,merrill,dh2023}.

One of the reasons for the popularity of this approach in computational physics compared to standard FEM-BEM coupling methods is its simplicity: 
To compute an approximation of the physically relevant solution in the interior domain,
only the solution of \emph{two} (sparse) linear FEM systems 
and the interpolation of \emph{one} boundary integral operator are required.
However, by construction, it is applicable only to linear transmission problems for which~\eqref{eq:poisson_interior} is the Poisson problem,
whereas standard FEM-BEM methods allow for more general elliptic problems in the interior.

In this work, we perform the rigorous numerical analysis of the hybrid FEM-BEM method for arbitrary right-hand sides in \eqref{eq:poisson_interior} and arbitrary jump conditions in~\eqref{eq:jump_cond1}--\eqref{eq:jump_cond2}).
First, we prove \textsl{a~priori} error estimates by verifying a simple C\'ea-type lemma.
Next, building on~\cite{afkpp13}, which thoroughly analyzes adaptive FEM for second-order elliptic PDEs with inhomogeneous Dirichlet boundary conditions, we derive a reliable and efficient \textsl{a~posteriori} error estimator.
Then, we use this to steer an adaptive algorithm of the standard form
\begin{equation*}
\mathrm{SOLVE}
\quad\rightarrow\quad
\mathrm{ESTIMATE}
\quad\rightarrow\quad
\mathrm{MARK}
\quad\rightarrow\quad
\mathrm{REFINE}.
\end{equation*}
Exploiting again ideas from~\cite{afkpp13} and thus of the seminal works~\cite{stevenson07,ckns08} on rate optimality of adaptive FEM
(see also the review article~\cite{cfpp14}),
we show that the adaptive algorithm converges at optimal rates with respect to the number of elements. 
In our convergence analysis, a crucial role is played by the local inverse estimates
for nonlocal boundary integral operators shown in~\cite{fkmp13}
(see also their generalization to curved boundaries established in~\cite{affkmp17}),
used in these papers to show rate optimality of adaptive BEM (see also the review article~\cite{ffhkp15}). 

We mention that for usual FEM-BEM coupling methods,
involving both FEM and BEM matrices in the overall linear system,
adaptivity has already been investigated in the pioneering work~\cite{cs1995}. 
In~\cite{affkmp13}, reliable error estimators for the three classical FEM-BEM coupling approaches~\cite{jn1980,bm1983,costabel1988} have been proposed and plain convergence of corresponding adaptive algorithms has been proved.
Optimal convergence rates have only recently been shown in~\cite{feischl17} for piecewise polynomial ansatz functions enriched with certain bubble functions and in~\cite{feischl22} for standard ansatz functions. 
The key challenge was the lack of some Pythagoras identity owing to the nonsymmetry of the couplings. 
Instead, for our adaptive hybrid FEM-BEM method,
we prove a quasi-orthogonality property similar to the one in~\cite{afkpp13}.

\subsection{Outline}

The remainder of this work is organized as follows:
In Section~\ref{sec:notation}, we collect some general notation used throughout the paper.
In Section~\ref{sec:preliminaries}, we recall the definition of Sobolev spaces in the interior, the exterior, and on the boundary along with standard discrete approximation subspaces. 
We further recall classical boundary integral operators and provide the weak formulation of the considered problem~\eqref{eq:transmission}.
In Section~\ref{sec:hybrid}, we formulate our hybrid FEM-BEM method and derive corresponding \textsl{a~priori} (Proposition~\ref{prop:apriori}) and \textsl{a~posteriori} (Proposition~\ref{prop:reliability}) error estimates.
Building on the latter, we present in Section~\ref{sec:adaptivity} an adaptive algorithm (Algorithm~\ref{alg:adaptive_algorithm}) and state our main result on optimal convergence (Theorem~\ref{thm:estimator_convergence}). 
The proof, which relies on the abstract framework of~\cite{cfpp14} and consists of the verification of the
\emph{axioms of adaptivity} for the error estimator, is presented in Section~\ref{sec:axioms}.
We conclude the paper in Section~\ref{sec:numerics} with numerical experiments,
demonstrating the practical performance of the method and of the adaptive algorithm.

\subsection{General notation} \label{sec:notation}
Throughout and without any ambiguity, $|\cdot|$ denotes the absolute value of scalars, the Euclidean norm of vectors in $\R^m$, or the measure of a set in $\R^m$, e.g., the length of an interval or the area of a surface in $\R^3$.
We write $A\lesssim B$ to abbreviate $A\le CB$ with some generic constant $C>0$ which is clear from the context.
Moreover, $A\eqsim B$ abbreviates $A\lesssim B\lesssim A$.

\section{Preliminaries}\label{sec:preliminaries}

\subsection{Sobolev spaces}
For measurable $\omega \subseteq \Omega$ or $\omega \subseteq \Gamma$, we abbreviate the corresponding $L^2$-norm by 
$\norm{\cdot}{\omega} := \norm{\cdot}{L^2(\omega)}$. 
Moreover, we abbreviate $H^1_*(\Omega) := \big\{ v \in H^1(\Omega) : \inner[\Omega]{v}{1} = 0 \big\}$. 
As usual, let $\widetilde H^{-1}(\Omega) := H^{1}(\Omega)^*$ and $H^{-1/2}(\Gamma) := H^{1/2}(\Gamma)^*$ be the dual spaces of $H^{1}(\Omega)$ and $H^{1/2}(\Gamma)$, respectively. 
We denote by $\inner[\Omega]{\cdot}{\cdot}$ and $\inner[\Gamma]{\cdot}{\cdot}$ the corresponding duality products,
and note that they coincide with $L^2$-scalar products if the arguments are in $L^2$.
For open sets $\omega\subseteq\R^d\setminus\Gamma$, we consider the Sobolev space $\Hloc^1(\omega):=\set{v:\omega\to \R}{v|_{\omega'} \in H^1(\omega') \text{ for all bounded open sets }\omega'\subseteq\omega}$.

We denote by $(\cdot)|_\Gamma: H^1(\Omega) \to H^{1/2}(\Gamma)$ the interior trace operator, 
which coincides with the usual restriction $v \vert_{\Gamma}$ for all continuous functions $v \in C(\overline\Omega)$. 
Similarly, we denote by $(\cdot)|_\Gamma: \Hloc^1(\Omegaext) \to H^{1/2}(\Gamma)$ also the exterior trace operator. 
If the context permits, we will omit the explicit notation $(\cdot)|_\Gamma$, e.g., we write $\norm{v}{H^{1/2}(\Gamma)}$ instead of $\norm{v|_\Gamma}{H^{1/2}(\Gamma)}$ for functions $v\in H^1(\Omega)$ or $v\in \Hloc^1(\Omegaext)$.
We denote by $\partial_{\n}: \{ v \in H^1(\Omega) : \Delta v \in \widetilde H^{-1}(\Omega) \} \to H^{-1/2}(\Gamma)$
the interior normal derivative, which coincides with the classical derivative $\partial_{\n} v$ for all smooth functions $v \in C^1(\overline\Omega)$. 
Similarly, we denote by $\partial_{\n}: \{ v \in \Hloc^1(\Omegaext) : \Delta v = 0 \} \to H^{-1/2}(\Gamma)$ also the exterior normal derivative.

\subsection{Integral operators} \label{eq:integral_operators}
We denote by $G \in C^{\infty} (\R^d \setminus \{ 0\})$ the Newtonian kernel defined, for $\z \in \R^d \setminus \{ 0\}$, by
\begin{equation*}
G(\z) =
\begin{cases}
- \frac{1}{2 \pi} \log\abs{\z} & \text{if } d=2,\\
\frac{1}{4 \pi} \frac{1}{\abs{\z}} & \text{if } d=3.\\
\end{cases}
\end{equation*}
For all sufficiently smooth $w : \Gamma \to \R$, we define the double-layer potential of $w$ as
\begin{equation} \label{eq:double_layer}
\widetilde K w(\x)
= \int_{\Gamma} \partial_{\n(\y)} G(\x - \y) \, w(\y) \, \dy
\quad \text{for } \x \in \R^d \setminus \Gamma.
\end{equation}
It is well known that this potential can be extended to a bounded and linear operator 
$\widetilde K: H^{1/2}(\Gamma) \to \Hloc^1(\R^d \setminus \Gamma)$,
and satisfies the Laplace equation
\begin{equation*}
-\Delta \widetilde K = 0 \quad \text{ in }\R^d\setminus\Gamma,
\end{equation*}
the jump conditions
\begin{equation*}
\big((\widetilde K w)|_\Omega - (\widetilde K w)|_{\Omegaext}\big)|_\Gamma = -w
\,\,\, \text{and} \,\,\,
\partial_{\n} \big((\widetilde K w)|_\Omega - (\widetilde K w)|_{\Omegaext}\big) = 0
\quad \text{for all } w \in H^{1/2}(\Gamma),
\end{equation*}
and the radiation condition
\begin{equation*}
\widetilde K w(\x) = c\log\abs{\x} + \OO(\abs{\x}^{-1}) \quad \text{as } \abs{\x} \to \infty
\end{equation*}
for some arbitrary $c\in\R$ if $d=2$ and $c=0$ if $d=3$.
We denote by $K: H^{1/2}(\Gamma) \to H^{1/2}(\Gamma)$ with $Kw := \big((\widetilde Kw)|_\Omega\big)|_\Gamma + w/2$ the double-layer operator.
For all sufficiently smooth $w : \Gamma \to \R$, $Kw$ satisfies the integral representation~\eqref{eq:double_layer} for almost all $\x\in\Gamma$. 
Finally, we mention that the restriction $K:H^1(\Gamma)\to H^1(\Gamma)$ onto $H^1(\Gamma)$ is a bounded and linear operator to $H^1(\Gamma)$. 
For details and proofs, we refer to the monographs~\cite{mclean00,steinbach08,ss11}. 

\subsection{Weak formulation}
We recall the weak formulation of~\eqref{eq:transmission}.
For the problem data, we assume that
$f \in \widetilde H^{-1}(\Omega)$,
$g \in H^{1/2}(\Gamma)$,
and $\phi \in H^{-1/2}(\Gamma)$.
The weak formulation then reads as follows:
Find $(u,\uext) \in H^1(\Omega) \times \Hloc^1(\Omegaext)$
with $(u - \uext)|_\Gamma = g$
such that
\begin{equation*}
\inner[\Omega]{\nabla u}{\nabla v}
+ \inner[\Omegaext]{\nabla \uext}{\nabla v}
=
\inner[\Omega]{f}{v}
+ \inner[\Gamma]{\phi}{v}
\quad
\text{for all }
v \in C^{\infty}_c(\R^d).
\end{equation*}
It is well known that there exists indeed a unique weak solution $(u,u^{\rm ext})$ to this problem; see again the monographs~\cite{mclean00,steinbach08,ss11}. 

\subsection{Discrete spaces}
We consider triangulations $\TT_\coarse$ of $\Omega$, i.e., sets of open $d$-dimensional simplices (triangles if $d=2$ and tetrahedra if $d=3$) forming a partition of $\Omega$ in the sense that $\overline\Omega = \bigcup_{T\in\TT_\coarse} \overline T$ and $T\cap T' = \emptyset$ for all $T\neq T'\in\TT_\coarse$.
Throughout, triangulations $\TT_\coarse$ are assumed to be conforming in the sense that $\overline T\cap \overline T'$ is either empty, a common vertex, a common edge, or a common face (if $d=3$) for all $T\neq T'\in\TT_\coarse$.
In particular, this induces a conforming triangulation $\TT_\coarse|_\Gamma$ of $\Gamma$ into (with respect to $\Gamma$) open $(d-1)$-dimensional simplices (line segments if $d=2$ and triangles if $d=3$).
For $p \in \N$ and $T \in \TT_\coarse$, we denote by $\PP^p(T)$ the space of polynomials of degree at most $p$ on $T$ and by $\PP^p(\TT_\coarse):=\set{v_h \in L^2(\Omega)}{v_h|_T \in \PP^p(T) \text{ for all }T\in\TT_\coarse}$ the space of $\TT_\coarse$-piecewise polynomials of degree at most $p$. 
We consider the space of globally continuous $\TT_\coarse$-piecewise polynomials
\begin{equation*}
	\SS^p(\TT_\coarse)
	:=C^0(\overline{\Omega}) \cap \PP^p(\TT_\coarse)
	\subset H^1(\Omega).
\end{equation*}
Let $J_\coarse^\Omega: H^1(\Omega) \to \SS^p(\TT_\coarse)$ denote the corresponding Scott--Zhang projection from~\cite{sz90} and $\Pi_\coarse^\Omega: L^2(\Omega)\to \PP^{p-1}(\TT_\coarse)$ the $L^2(\Gamma)$-orthogonal projection. 
Additionally, we set 
$\SS^p_*(\TT_\coarse) := \big\{ v_\coarse \in \SS^p(\TT_\coarse) : \inner[\Omega]{v_\coarse}{1} = 0 \big\}$ 
and
$\SS^p_0(\TT_\coarse) := \big\{ v_\coarse \in \SS^p(\TT_\coarse) : v_\coarse \vert_{\Gamma} = 0 \big\}$.
We define $\PP^p(\TT_\coarse|_\Gamma)$ and $\SS^p(\TT_\coarse|_\Gamma)$ with corresponding Scott--Zhang projection $J_\coarse^\Gamma: H^1(\Gamma) \to \SS^p(\TT_\coarse|_\Gamma)$ analogously; see also~\cite{sv06}.
As the definition of $J_\coarse^\Gamma$ only involves integrals on edges for $d=2$ and faces for $d=3$, respectively, it is even well defined and stable on $L^2(\Gamma)$.  
By standard interpolation theory, we also see that $J_\coarse^\Gamma$ is $H^{1/2}$-stable, where the stability constant depends only on the boundary $\Gamma$, the shape-regularity of $\TT_\coarse$, and the polynomial degree $p$. 
We also mention the identity 
\begin{equation}\label{eq:sz_restricted}
	(J_\coarse^\Omega v)|_\Gamma = J_\coarse^\Gamma (v|_\Gamma) \quad\text{for all }v \in H^1(\Omega).
\end{equation}
Moreover, we require the $L^2(\Gamma)$-orthogonal projection $\Pi^\Gamma_\coarse:L^2(\Gamma) \to \PP^{p-1}(\TT_\coarse|_\Gamma)$. 

\section{Hybrid FEM-BEM method}\label{sec:hybrid}
The hybrid FEM-BEM method relies on the superposition principle, i.e., we consider some decomposition $u=u_1 + u_2$ in the interior domain $\Omega$. 
We additionally suppose the compatibility condition
\begin{equation} \label{eq:compatibility}
	\inner[\Omega]{f}{1} + \inner[\Gamma]{\phi}{1} = 0;
\end{equation}
see also Remark~\ref{rem:compatibility}. 
Then, there exists a unique weak solution $u_1 \in H^1_*(\Omega)$ of the Neumann problem
\begin{subequations} \label{eq:fk1}
\begin{alignat}{2}
- \Delta u_1 &= f &\quad& \text{in } \Omega,\\
\partial_{\n} u_1 &= \phi && \text{on } \Gamma.
\end{alignat}
\end{subequations}
Define $u_2 := u - u_1$ in $\Omega$.
Then, by construction, $(u_2,\uext)$ is the unique weak solution of the full-space transmission problem
\begin{alignat*}{2}
- \Delta u_2 &= 0 &\quad& \text{in } \Omega,\\
- \Delta \uext &= 0 &\quad& \text{in } \Omegaext,\\
u_2 - \uext &= g - u_1 && \text{on } \Gamma,\\
\partial_{\n}( u_2 -  \uext) &= 0 && \text{on } \Gamma,\\
\uext(\x) &= c\log\abs{\x} + \OO(\abs{\x}^{-1}) && \text{as } \abs{\x} \to \infty.
\end{alignat*}
Hence, $u_2 = (\widetilde K(u_1 - g))|_\Omega$ and $\uext = (\widetilde K(u_1 - g))|_{\Omegaext}$; see Section~\ref{eq:integral_operators}.
In particular, $u_2$ can be characterized as the solution of the Dirichlet problem
\begin{subequations} \label{eq:fk2}
\begin{alignat}{2}
- \Delta u_2 &= 0 &\quad& \text{in } \Omega,\\
u_2  &= (K-1/2)(u_1 - g) && \text{on } \Gamma.
\end{alignat}
\end{subequations}

Given a conforming triangulation $\TT_\coarse$ and a polynomial degree $p$, this suggests the following discretization: 
\begin{itemize}
\item[(i)]
Find $u_{1,\coarse} \in \SS^p_*(\TT_\coarse)$
such that
\begin{equation}\label{eq:u1_coarse}
\inner[\Omega]{\nabla u_{1,\coarse}}{\nabla v_\coarse}
= \inner[\Omega]{f}{v_\coarse}
+ \inner[\Gamma]{\phi}{v_\coarse}
\quad
\text{for all } v_\coarse \in \SS^p_*(\TT_\coarse).
\end{equation}
\item[(ii)]
Find $u_{2,\coarse} \in \SS^p(\TT_\coarse)$ with $u_{2,\coarse}\vert_{\Gamma} = J_\coarse^\Gamma (K-1/2)(u_{1,\coarse} - g)$
such that
\begin{equation}\label{eq:u2_coarse}
\inner[\Omega]{\nabla u_{2,\coarse}}{\nabla v_\coarse}
= 0
\quad
\text{for all } v_\coarse \in \SS^p_0(\TT_\coarse).
\end{equation}
\item[(iii)] Define $u_\coarse := u_{1,\coarse} + u_{2,\coarse} \in \SS^p(\TT_\coarse)$.
\end{itemize}

\begin{remark}\label{rem:compatibility} 
Integration by parts for $\inner[\Omega]{f}{1} = -\inner[\Omega]{\Delta u}{1}$ shows that
\begin{align*}
	\inner[\Omega]{f}{1} + \inner[\Gamma]{\phi}{1} 
	= \inner[\Omega]{\nabla u}{\nabla 1} + \inner[\Gamma]{\phi-\partial_{\n}u}{1} 
	= -\inner[\Gamma]{\partial_{\n}u^{\rm ext}}{1}. 
\end{align*}
The compatibility condition~\eqref{eq:compatibility} can thus always be guaranteed by replacing $\uext$ by $\uext - \widetilde u^{\rm ext}$, where $\widetilde u^{\rm ext}$ can be chosen as  arbitrary function $\widetilde u^{\rm ext}$ satisfying
\begin{subequations}\label{eq:compatibility_shift}
\begin{align}\label{eq:compatibility_shift_laplace}
	- \Delta \widetilde u^{\rm ext} &= 0 \text{ in } \Omegaext,\\ \label{eq:compatibility_shift_product}
	\inner[\Gamma]{\partial_{\n} \widetilde u^{\rm ext}}{1} &= -\inner[\Gamma]{\partial_{\n}u^{\rm ext}}{1} = \inner[\Omega]{f}{1} + \inner[\Gamma]{\phi}{1},\\ \label{eq:compatibility_shift_radiation}
	\widetilde u^{\rm ext}(\x) &= \widetilde c \log\abs{\x} + \OO(\abs{\x}^{-1})\text{ as } \abs{\x} \to \infty,
\end{align}
\end{subequations}
for some $\widetilde c\in\R$ if $d=2$ and $\widetilde c=0$ if $d=3$. 
Then $(u, u^{\rm ext} - \widetilde u^{\rm ext})$ satisfies the transmission problem~\eqref{eq:transmission} with $g$ replaced by $g - \widetilde u^{\rm ext}|_\Gamma$ and $\phi$ replaced by $\phi - \partial_{\n} \widetilde u^{\rm ext}$. 
By construction, the modified problem satisfies the compatibility condition $\inner[\Omega]{f}{1} + \inner[\Gamma]{\phi - \partial_{\n} \widetilde u^{\rm ext}}{1} = 0$. 

Finally, we note that solutions to~\eqref{eq:compatibility_shift} can be easily constructed by means of the Newtonian kernel.
To this end, abbreviate $C:=\inner[\Omega]{f}{1} + \inner[\Gamma]{\phi}{1}$, fix some $\x \in \Omega$, and set 
\begin{align*}
	\widetilde u^{\rm ext}(\y) := -C \,G(\x-\y) \quad \text{for all }\y \in \Omegaext.
\end{align*}
This function satisfies \eqref{eq:compatibility_shift_laplace} and \eqref{eq:compatibility_shift_radiation}; see, e.g., \cite[Lemma~7.13]{mclean00}. 
To see~\eqref{eq:compatibility_shift_product}, we rewrite
\begin{align*}
	\inner[\Gamma]{\partial_{\n} \widetilde u^{\rm ext}}{1} 
	= -C \int_\Gamma \partial_{\n(\y)} G(\x-\y) \, \d\y
	= -C\,\widetilde K 1(\x).
\end{align*}
By the representation formula applied to the constant function $1$ (see, e.g., \cite[Theorem~6.10]{mclean00}), it holds that $\widetilde K 1(\x)=-1$ so that the last term indeed just equals $C$.
\end{remark}

\begin{remark}
It is well known that the solution $u_1$ of the Neumann problem~\eqref{eq:fk1}
is unique up to an additive constant.
In our method, to fix the ideas, we impose uniqueness of $u_1$ by requiring
that it has zero integral mean.
However, we note that this choice is arbitrary and that the approximation of $u$
is independent of this choice.
To see this,
let $\widetilde u_1 \in H^1(\Omega)$ be the unique solution of the Neumann problem~\eqref{eq:fk1}
obtained by fixing a different additive constant, i.e.,
$\widetilde u_1 = u_1 + \widetilde c$ for some $\widetilde c \in \R$.
Then, define  $\widetilde u_2 := \widetilde K(\widetilde u_1 - g)$ and $\widetilde u := \widetilde u_1 + \widetilde u_2$.
Since $\widetilde K1 = -1$,
we infer that
\begin{align*}
	\widetilde u_2 = \widetilde K(\widetilde u_1 - g) 
	= \widetilde{K}(u_1 + \widetilde c - g) 
	= \widetilde K (u_1  - g) + \widetilde K\widetilde c = u_2 - \widetilde c.
\end{align*}
We thus obtain that
\begin{align*}
	\widetilde u 
	= \widetilde u_1 + \widetilde u_2 
	= u_1 + \widetilde c + u_2 - \widetilde{c}
	= u_1 + u_2 = u.
\end{align*}
\end{remark}

\subsection{\textsl{A~priori} error estimation}
The next proposition shows that the pair $\u_\coarse:=(u_{1,\coarse},u_{2,\coarse})$ is indeed a quasi-best approximation of $\u:=(u_1,u_2)$. 

\begin{proposition}\label{prop:apriori}
There exists $\const{c\'ea}>0$ such that 
\begin{equation}\label{eq:cea}
	\norm{\u - \u_\coarse}{H^1(\Omega)} \le \const{c\'ea} \min_{\v_\coarse \in \SS^p(\TT_\coarse)^2} \norm{\u - \v_\coarse}{H^1(\Omega)}.
\end{equation}
The constant $\const{c\'ea}$ depends only on the domain $\Omega$, the shape-regularity of $\TT_\coarse$, and the polynomial degree $p$. 
\end{proposition}
\begin{proof}
The Poincar\'e inequality and the fact that $u_{1,\coarse}$ is the orthogonal projection of $u_1$ onto the space $\SS^p_*(\TT_\coarse)$ with respect to the norm $\norm{\nabla(\cdot)}{\Omega}$ show that 
\begin{equation}\label{eq:cea1}
	\norm{u_1 - u_{1,\coarse}}{H^1(\Omega)} 
	\lesssim \norm{\nabla(u_1 - u_{1,\coarse})}{\Omega} 
	= \min_{v_\coarse \in \SS^p(\TT_\coarse)} \norm{\nabla(u_1 - v_\coarse)}{\Omega}.
\end{equation}

For the second term, we require the harmonic lifting operator $L:H^{1/2}(\Gamma)\to H^1(\Omega)$, mapping $w\in H^{1/2}(\Gamma)$ to the unique solution $Lw \in H^1(\Omega)$ of the Dirichlet problem
\begin{subequations} \label{eq:lifting}
\begin{alignat}{2}
- \Delta Lw &= 0 &\quad& \text{in } \Omega,\\
Lw  &= w && \text{on } \Gamma.
\end{alignat}
\end{subequations}
In particular, we have that $u_2 = L(K-1/2)(u_1 - g)$. 
We introduce $u_2^\coarse := L(K-1/2) (u_{1,\coarse} - g)$ and apply the triangle inequality to see that
\begin{equation*}
	\norm{u_2 - u_{2,\coarse}}{H^1(\Omega)} 
	 \le \norm{u_2 - u_2^\coarse}{H^1(\Omega)} + \norm{u_2^\coarse - u_{2,\coarse}}{H^1(\Omega)}.
\end{equation*}
Stability of $L$, $K$, and the trace operator show that 
\begin{equation}\label{eq:aux_efficiency}
	\norm{u_2 - u_2^\coarse}{H^1(\Omega)}
	\lesssim \norm{(K-1/2) (u_1 - u_{1,\coarse}) }{H^{1/2}(\Gamma)} 
	\lesssim \norm{ u_1 - u_{1,\coarse} }{H^{1}(\Omega)}.
\end{equation}
Noting that $u_{2,\coarse}$ is the Galerkin approximation to $u_2^\coarse$ with $u_{2,\coarse}|_\Gamma = J_\coarse^\Gamma u_2^\coarse$,
the quasi-best-approximation result \cite[Proposition~2.4]{afkpp13} and the triangle inequality show that
\begin{align*}
	\norm{u_2^\coarse - u_{2,\coarse}}{H^1(\Omega)} 
	\lesssim \min_{v_\coarse \in \SS^p(\TT_\coarse)} \norm{u_2^\coarse - v_\coarse}{H^1(\Omega)}
	\le \norm{u_2 - u_2^\coarse}{H^1(\Omega)} +  \min_{v_\coarse \in \SS^p(\TT_\coarse)} \norm{u_2 - v_\coarse}{H^1(\Omega)}.
\end{align*}
Overall, \eqref{eq:cea1} and \eqref{eq:aux_efficiency} conclude the proof.

The remaining estimate 
\begin{equation*}
	\norm{u_2 - u_{2,\coarse}}{H^1(\Omega)} \lesssim \min_{v_\coarse \in \SS^p(\TT_\coarse)} \norm{\nabla(u_2 - v_\coarse)}{\Omega}
\end{equation*}
follows from \cite[Proposition~2.3]{afkpp13}. 
\end{proof}

From~\eqref{eq:cea} and standard approximation theory, it follows for the solution $u$ of~\eqref{eq:transmission} in the interior domain that 
\begin{align}\label{eq:a_priori}
	\norm{u - u_\coarse}{H^1(\Omega)} \lesssim \max_{T\in\TT_\coarse} \, \diam(T)^p
\end{align}
if the decomposition $u = u_1 + u_2$ is such that $u_1,u_2 \in H^{p+1}(\Omega)$.

\begin{remark}\label{rem:cea}
Note that~\eqref{eq:cea} only guarantees that $(u_{1,\coarse},u_{2,\coarse})$ is a quasi-best approximation of $(u_1,u_2)$, and not necessarily that the sum $u_\coarse=u_{1,\coarse}+u_{2,\coarse}$ is a quasi-best approximation of $u = u_1+u_2$.
In particular, the \textsl{a~priori} estimate~\eqref{eq:a_priori} requires regularity of both $u_1$ and $u_2$.
In contrast to this, the standard FEM-BEM couplings~\cite{jn1980,costabel1988} provide quasi-best approximations of the pair $(u,\partial_{\n} \uext)$.
\end{remark}

\subsection{\textsl{A~posteriori} error estimation}

Assuming the additional regularities $f\in L^2(\Omega)$, $\phi \in L^2(\Omega)$, and $g \in H^1(\Gamma)$, we define for all elements $T\in\TT_\coarse$ with mesh-size $h_T := |T|^{1/d}$ the error indicators
\begin{subequations}\label{eq:estimator}
\begin{align}\label{eq:estimator1}
	\eta_{1,\coarse}(T)^2 := h_T^2 \norm{(f+\Delta u_{1,\coarse})}{T}^2 + h_T \norm{ [\![\partial_{\n} u_{1,\coarse}]\!]}{\partial T\cap \Omega}^2 + h_T \norm{\phi - \partial_{\n} u_{1,\coarse}}{\partial T\cap \Gamma}^2
\end{align}
as well as 
\begin{align}\label{eq:estimator2}
	\eta_{2,\coarse}(T)^2 := h_T^2 \norm{\Delta u_{2,\coarse}}{T}^2 +  h_T \norm{ [\![ \partial_{\n} u_{2,\coarse}]\!] }{\partial T\cap \Omega}^2 + h_T \norm{(1-\Pi^\Gamma_\coarse)\nabla_\Gamma (K-1/2) (u_{1,\coarse} - g)}{\partial T\cap \Gamma}^2.
\end{align}
Here, $\n$ denotes the outward-pointing unit normal vector $\n$ on $\partial T$ and  $[\![ \partial_{\n} (\cdot)]\!]$ denotes as usual the normal jump. 
Note that the final term in \eqref{eq:estimator2} is well-defined, as $K:H^1(\Gamma)\to H^1(\Gamma)$ (see Section~\ref{eq:integral_operators}) and $u_{1,\coarse} - g \in H^1(\Gamma)$. 
We also define the combined indicator 
\begin{align}
	\eta_\coarse(T)^2 := \eta_{1,\coarse}(T)^2 + \eta_{2,\coarse}(T)^2. 
\end{align}
The corresponding error estimators read as
\begin{align}
	\eta_{1,\coarse}^2 := \sum_{T\in\TT_\coarse} \eta_{1,\coarse}(T)^2,
	\quad
	\eta_{2,\coarse}^2 := \sum_{T\in\TT_\coarse} \eta_{2,\coarse}(T)^2,
	\quad
	\eta_{\coarse}^2 := \sum_{T\in\TT_\coarse} \eta_{\coarse}(T)^2 = \eta_{1,\coarse}^2 + \eta_{2,\coarse}^2.
\end{align}
\end{subequations}

In the following proposition, we show that the combined estimator $\eta_\coarse$ is indeed reliable and efficient up to oscillations.
We define the local oscillations associated to the right-hand side $f$, the Neumann datum $\phi$, and the Dirichlet datum $g$ by
\begin{align*}
	\osc_{\rm RHS,\coarse}(T)^2 
	&:= h_T^2 \norm{(1-\Pi_\coarse^\Omega)f}{T}^2,
	\\
	\osc_{\rm N,\coarse}(T)^2 
	&:= h_T \norm{(1-\Pi^\Gamma_\coarse)\phi}{\partial T\cap \Gamma}^2,
	\\
	\osc_{\rm D,\coarse}(T)^2 
	&:= h_T \norm{(1-\Pi^\Gamma_\coarse)\nabla_\Gamma (K-1/2) (u_{1,\coarse} - g)}{\partial T\cap \Gamma}^2
\end{align*}
for all $T\in\TT_\coarse$ as well as the corresponding oscillations 
\begin{equation*}
	\osc_{\rm RHS,\coarse}^2 
	:= \sum_{T\in\TT_\coarse} \osc_{\rm RHS,\coarse}(T)^2, 
	\quad\osc_{\rm N,\coarse}^2 
	:= \sum_{T\in\TT_\coarse} \osc_{\rm N,\coarse}(T)^2, 
	\quad\osc_{\rm D,\coarse}^2 
	:= \sum_{T\in\TT_\coarse} \osc_{\rm D,\coarse}(T)^2.
\end{equation*}
Note that 
\begin{equation}\label{eq:osc_bound}
	\max \big\{\osc_{\rm RHS,\coarse}, \osc_{\rm N,\coarse}, \osc_{\rm D,\coarse} \big\} \le \eta_\coarse.
\end{equation}

\begin{proposition} \label{prop:reliability}
There exist $\const{rel}>0$ and $\const{eff}>0$ such that
\begin{equation}\label{eq:reliability_efficiency}
	\const{rel}^{-1} \norm{\u - \u_\coarse}{H^1(\Omega)} \le \eta_\coarse \le \const{eff} \big(\norm{\nabla (\u - \u_\coarse)}{\Omega} + \osc_{\rm RHS,\coarse} + \osc_{\rm N,\coarse} + \osc_{\rm D,\coarse} \big). 
\end{equation}
The constant $\const{rel}$ depends only on the domain $\Omega$, the shape-regularity of $\TT_\coarse$, the shapes of the boundary element patches $\bigcup\set{T\in\TT_\coarse|_\Gamma}{\overline T\cap \overline {T'} \neq \emptyset}$ for $T'\in\TT_\coarse|_\Gamma$, and the polynomial degree $p$. 
Notice that, if $\TT_\coarse$ is obtained from iterative newest vertex bisection of some initial conforming triangulation $\TT_0$, the number of different patch shapes is uniformly bounded; see, e.g., \cite[Proof of Proposition~3.1]{afkpp13}. 
The constant $\const{eff}$ depends only on the domain $\Omega$, the shape-regularity of $\TT_\coarse$, and the polynomial degree $p$. 
\end{proposition}
\begin{proof}
Standard arguments readily yield reliability and efficiency for the first term of the estimator, i.e.,
\begin{align*}
	\norm{u_1 - u_{1,\coarse}}{H^1(\Omega)} 
	\lesssim \eta_{1,\coarse} 
	\lesssim \norm{\nabla (u_1 - u_{1,\coarse})}{\Omega} + \osc_{\rm RHS,\coarse} + \osc_{\rm N,\coarse};
\end{align*}
see, e.g., \cite{ao11,verfuerth13} for details. 

For the second term, we recall $u_2^\coarse := L(K-1/2) (u_{1,\coarse} - g)$ from the proof of Proposition~\ref{prop:apriori}.
Then, \eqref{eq:aux_efficiency} as well as reliability for $u_1$ show that
\begin{equation*}
	\norm{u_2 - u_2^\coarse}{H^1(\Omega)}
	\lesssim \norm{ u_1 - u_{1,\coarse} }{H^{1}(\Omega)}
	\lesssim \eta_{1,\coarse}.
\end{equation*}

Noting that $u_{2,\coarse}$ is the Galerkin approximation to $u_2^\coarse$ with $u_{2,\coarse}|_\Gamma = J_\coarse^\Gamma u_2^\coarse$, we have that $\norm{u_2^\coarse - u_{2,\coarse}}{H^1(\Omega)} \lesssim \eta_{2,\coarse}$ thanks to the reliability estimate from \cite[Proposition~2.4]{afkpp13}.
Overall this shows the first inequality in~\eqref{eq:reliability_efficiency}. 

On the other hand, the efficiency estimate from~\cite[Proposition~2.4]{afkpp13} shows that
\begin{align*}
	\eta_{2,\coarse} 
	\lesssim 
	\norm{\nabla(u_2^\coarse - u_{2,\coarse})}{\Omega} + \osc_{\rm D,\coarse}.
\end{align*}
The triangle inequality further shows that 
\begin{equation*}
	\norm{\nabla (u_2^\coarse - u_{2,\coarse})}{\Omega}^2 
	\le \norm{\nabla (u_2 - u_{2,\coarse})}{\Omega} + \norm{\nabla (u_2 - u_2^\coarse)}{\Omega}.
\end{equation*}
The second inequality in~\eqref{eq:reliability_efficiency} then follows from~\eqref{eq:aux_efficiency} together with the~Poincar\'e inequality. 
\end{proof}

\section{Adaptive algorithm \& Optimal convergence}\label{sec:adaptivity}

\subsection{Adaptive algorithm}
Let $\TT_0$ be some given conforming initial triangulation of $\Omega$. 
We employ newest vertex bisection~\cite{stevenson08} as mesh-refinement strategy and abbreviate the set of all conforming triangulations that can be reached from $\TT_0$ by $\T$. 
For each $\TT_\coarse\in\T$ and marked elements $\MM_\coarse\subseteq\TT_\coarse$, let $\TT_\fine := \refine(\TT_\coarse,\MM_\coarse)$ be the coarsest conforming triangulation, where all $T\in\TT_\coarse$ have been refined, i.e., $\MM_\coarse \subseteq \TT_\coarse\setminus\TT_\fine$. 
We write $\TT_\fine \in \T(\TT_\coarse)$, if $\TT_\fine$ results from $\TT_\coarse$ by finitely many steps of refinement. 
In particular, we have that $\T = \T(\TT_0)$. 

We consider the following standard adaptive algorithm using D\"orfler marking. 

\begin{algorithm} \label{alg:adaptive_algorithm}
\textbf{Input:} Initial conforming triangulation $\TT_0$, polynomial degree $p$, D\"orfler parameter $0 < \theta \le 1$. \\
\textbf{Loop:} For each $\ell=0,1,2,\dots$, iterate the following steps~{\rm(i)--(iv)}:
\begin{itemize}
\item[\rm(i)] {\bf SOLVE:} Compute the discrete solutions $u_{1,\ell} \in \SS_*^p(\TT_\ell)$ of \eqref{eq:u1_coarse} and $u_{2,\ell} \in \SS^p(\TT_\ell)$ of \eqref{eq:u2_coarse}. 
\item[\rm(ii)] {\bf ESTIMATE:} Compute the error indicators $\eta_\ell(T)$ of~\eqref{eq:estimator} for all ${T}\in\TT_\ell$. 
\item[\rm(iii)] {\bf MARK:} Determine a minimal set of marked elements $\MM_\ell\subseteq\TT_\ell$ such that 
\begin{equation*}
	\theta \sum_{T\in\TT_\ell} \eta_\ell(T)^2 \le \sum_{T\in\MM_\ell} \eta_\ell(T)^2.
\end{equation*} 
\item[\rm(iv)] {\bf REFINE:} Generate the refined mesh $\TT_{\ell+1}:=\refine(\TT_\ell,\MM_\ell)$. 
\end{itemize}
\textbf{Output:} Refined meshes $\TT_\ell$, corresponding discrete solutions $\u_\ell = (u_{1,\ell}, u_{2,\ell})$, and 
error estimators $\eta_\ell$ for all $\ell \in \N_0$.\qed
\end{algorithm}

\subsection{Optimal convergence}
For $s>0$, we define 
\begin{equation*}
	\const{apx}^\eta(s) := \sup_{N\in\N_0} \min_{\TT_\coarse\in\T_N} (N+1)^s \eta_\coarse \in [0,\infty]
	\quad \text{with} \quad \T_N := \set{\TT_\coarse\in \T}{\#\TT_\coarse - \#\TT_0 \le N}.
\end{equation*}
By definition, $\const{apx}^\eta(s)<\infty$ implies that the error estimator $\eta_\coarse$ decays at least with rate $\OO((\#\TT_\coarse)^{-s})$ on the optimal meshes $\TT_\coarse$. 
The following main theorem states that each possible rate $s>0$ is indeed realized by Algorithm~\ref{alg:adaptive_algorithm}. 

\begin{theorem}\label{thm:estimator_convergence}
For arbitrary $0<\theta\le 1$, there exist $\const{lin}>0$ and $0<\rho_{\rm lin}<1$ such that 
\begin{equation}
	\eta_{\ell+j}^2 
	\le \const{lin} \rho_{\rm lin}^j \eta_\ell^2 \quad \text{for all }\ell,j\in\N_0.
\end{equation}
Moreover, there exists $0<\theta_\star<1$ such that for all $0<\theta<\theta_\star$ and all $s>0$, there exists $c_{\rm opt}, \const{opt}>0$ with 
\begin{equation}
	c_{\rm opt} \const{apx}^\eta(s) 
	\le \sup_{\ell\in\N_0} (\#\TT_\ell - \#\TT_0 + 1)^s \eta_\ell
	\le \const{opt} \const{apx}^\eta(s).
\end{equation}
The constant $\theta_\star$ depends only the domain $\Omega$, the initial triangulation $\TT_0$, and the used polynomial degree $p$, the constants $\const{lin}, \rho_{\rm lin}$ depend additionally on the D\"orfler parameter $\theta$, 
and the constant $\const{opt}$ depends additionally on $s$, while $c_{\rm opt}$ depends only on $\TT_0$. 
\end{theorem}

\begin{proof}
According to \cite[Theorem~4.1]{cfpp14}, the theorem follows if we can prove the so-called \emph{axioms of adaptivity} for the error estimator, namely stability on nonrefined elements, stated and verified in Section~\ref{sec:stability}, reduction on refined elements, stated and verified in Section~\ref{sec:reduction}, discrete reliability, stated and verified in Section~\ref{sec:discrete_reliability}, and general quasi-orthogonality, stated and verified in Section~\ref{sec:orthogonality}. 
In particular, we choose the error measure from \cite[Section~2.2]{cfpp14} as ${\rm d}[\TT_\coarse;\boldsymbol{v},\boldsymbol{w}] := \norm{\nabla (\boldsymbol{v} - \boldsymbol{w})}{\Omega}$ for all $\TT_\coarse\in\T$ and all $\boldsymbol{v},\boldsymbol{w} \in H_*^1(\Omega) \times H^1(\Omega)$. 
The required mesh-refinement properties from~\cite[Section~2.4]{cfpp14} are satisfied according to~\cite[Section~2.5]{cfpp14} for the considered newest vertex bisection.
\end{proof}

\begin{remark}
Note that the estimator is equivalent to the total error, i.e.,
\begin{equation*}
	\eta_\coarse \eqsim \norm{\u - \u_\coarse}{H^1(\Omega)} + \osc_{\rm RHS,\coarse} + \osc_{\rm N,\coarse} + \osc_{\rm D,\coarse},
\end{equation*}
which is a direct consequence of~\eqref{eq:osc_bound} and \eqref{eq:reliability_efficiency}.
In particular, Theorem~\ref{thm:estimator_convergence} holds analogously for the total error instead of the estimator.

Instead, \cite{feischl22} proves optimal convergence of a standard adaptive algorithm employing the Johnson--N\'ed\'elec coupling~\cite{jn1980}, which approximates $(u,\partial_{\n}\uext)$ rather than $(u_1,u_2)$; cf.~Remark~\ref{rem:cea}. 
More precisely, optimal convergence of an estimator that measures the discretization error of $(u,\partial_{\n}\uext)$ is shown. 
It is only known that this estimator is an upper bound, but not necessarily a lower bound (not even up to oscillations), for this error.
\end{remark}

\section{Axioms of adaptivity}\label{sec:axioms}

In this section, we state and verify the four axioms of adaptivity from \cite{cfpp14} required to prove Theorem~\ref{thm:estimator_convergence}. 
We note that \cite{feischl22} verifies these axioms in case of the Johnson--N\'ed\'elec coupling~\cite{jn1980}. 
As for the Johnson--N\'ed\'elec coupling (and standard BEM~\cite{ffhkp15}), the crucial ingredient for the first two axioms, stability and reduction, are local inverse estimates for (non-local) boundary integral operators~\cite{affkmp17}, and the crucial ingredient for the third axiom, discrete reliability, are properties of the Scott--Zhang projection. 
For the proof of discrete reliability, we can particularly employ~\cite{afkpp13}, which treats the numerical approximation of the Poisson model problem with inhomogeneous Dirichlet boundary conditions discretized via the Scott--Zhang projection. 
We also took inspiration from \cite{afkpp13} to prove the fourth axiom, general quasi-orthogonality. 
However, \cite{afkpp13} only deals with the projection of fixed Dirichlet data, whereas the Dirichlet data~$(K-1/2)(u_{1,\coarse} - g)$ that are projected onto $\SS^p(\TT_\coarse|_\Gamma)$ for the computation of $u_{2,\coarse}\vert_{\Gamma}$ change in every step of the adaptive algorithm. 
This poses an additional technical challenge. 
Indeed, \cite{afkpp13} could even prove the quasi-orthogonality~\eqref{eq:quasi_orthogonality} below with $\varepsilon_2 = 0$. 
We can only show \eqref{eq:quasi_orthogonality} for some $\varepsilon_2 > 0$, from which we then derive the desired general quasi-orthogonality. 

For the last axiom, \cite{feischl22} uses a completely different route that covers general uniform inf-sup stable problems, including the Johnson-N\'ed\'elec coupling. 
The developed abstract framework is not applicable for the hybrid FEM-BEM method~\eqref{eq:u1_coarse}--\eqref{eq:u2_coarse}, 
as $(u_{1,\coarse},u_{2,\coarse})$ does not necessarily correspond to the Galerkin discretization of some inf-sup stable problem on the continuous level.

\subsection{Stability on nonrefined elements}\label{sec:stability}

We show stability on nonrefined elements, i.e., there exists a uniform constant $\const{stab}>0$ depending only on the domain $\Omega$, the shape-regularity of the considered triangulations (and thus on $\TT_0$), and the used polynomial degree $p$ such that for all refinements $\TT_\fine\in\T(\TT_\coarse)$ of a conforming triangulation $\TT_\coarse\in\T$, it holds that 
\begin{align}
	\Big| \Big(\sum_{T\in\TT_\coarse\cap\TT_\fine} \eta_\fine(T)^2\Big)^{1/2} - \Big(\sum_{T\in\TT_\coarse\cap\TT_\fine} \eta_\coarse(T)^2\Big)^{1/2} \Big|
	\le 
	\const{stab} \norm{\nabla (\u_\fine - \u_\coarse) }{\Omega}.
\end{align}

The inverse triangle inequality and the inequality $(a^2 + b^2)^{1/2} \le a+b$ for all $a,b\ge 0$ give that 
\begin{align*}
	&\Big| \Big(\sum_{T\in\TT_\coarse\cap\TT_\fine} \eta_\fine(T)^2\Big)^{1/2} - \Big(\sum_{T\in\TT_\coarse\cap\TT_\fine} \eta_\coarse(T)^2\Big)^{1/2} \Big|
	\\
	&\quad \le 
	\Big| \Big(\sum_{T\in\TT_\coarse\cap\TT_\fine} \eta_{1,\fine}(T)^2\Big)^{1/2} - \Big(\sum_{T\in\TT_\coarse\cap\TT_\fine} \eta_{1,\coarse}(T)^2\Big)^{1/2} \Big|
	\\
	&\qquad+\Big| \Big(\sum_{T\in\TT_\coarse\cap\TT_\fine} \eta_{2,\fine}(T)^2\Big)^{1/2} - \Big(\sum_{T\in\TT_\coarse\cap\TT_\fine} \eta_{2,\coarse}(T)^2\Big)^{1/2} \Big|.
\end{align*}

Standard arguments show stability for the first term, i.e., 
\begin{align*}
	\Big| \Big(\sum_{T\in\TT_\coarse\cap\TT_\fine} \eta_{1,\fine}(T)^2\Big)^{1/2} - \Big(\sum_{T\in\TT_\coarse\cap\TT_\fine} \eta_{1,\coarse}(T)^2\Big)^{1/2} \Big|
	\lesssim \norm{\nabla (u_{1,\fine} - u_{1,\coarse})}{\Omega};
\end{align*}
see, e.g., \cite{ckns08}.

For the the second term, we apply the inverse triangle inequality (twice) and the fact that $\Pi^\Gamma_\coarse = \Pi^\Gamma_\fine$ on $T\in\ \TT_\coarse\cap \TT_\fine$ to see that
\begin{align*}
	&\Big| \Big(\sum_{T\in\TT_\coarse\cap\TT_\fine} \eta_{2,\fine}(T)^2\Big)^{1/2} - \Big(\sum_{T\in\TT_\coarse\cap\TT_\fine} \eta_{2,\coarse}(T)^2\Big)^{1/2} \Big|
	\le 
	\Big(\sum_{T\in\TT_\coarse\cap\TT_\fine} \big| \eta_{2,\fine}(T) - \eta_{2,\coarse}(T) \big|^2\Big)^{1/2}
	\\
	&\le 
	\Big(\sum_{T\in\TT_\coarse\cap\TT_\fine} h_T^2 \norm{\Delta (u_{2,\fine} - u_{2,\coarse})}{T}^2 + h_T \norm{ [\![ \partial_{\n} (u_{2,\fine} - u_{2,\coarse})]\!]}{\partial T\cap \Omega}^2 \\
	&\qquad + h_T \norm{(1-\Pi^\Gamma_\coarse)\nabla_\Gamma (K-1/2) (u_{1,\fine}-u_{1,\coarse})}{\partial T\cap \Gamma}^2\Big)^{1/2}. 
\end{align*}
Since $(K-1/2)1 = -1$ is constant, we can replace the difference $u_{2,\fine} - u_{2,\coarse}$ on the right-hand side by $u_{2,\fine} - u_{2,\coarse} - c$ with $c:=|\Omega|^{-1} \int_\Omega (u_{2,\fine} - u_{2,\coarse}) \d\x$.
Then, an inverse inequality, a trace inequality, and the Poincar\'e inequality show for the first two terms that 
\begin{align*}
	\sum_{T\in\TT_\coarse\cap\TT_\fine} h_T^2 \norm{\Delta (u_{2,\fine} - u_{2,\coarse})}{T}^2 + h_T \norm{ [\![ \partial_{\n} (u_{2,\fine} - u_{2,\coarse})]\!]}{\partial T\cap \Omega}^2 
	\lesssim 
	\norm{\nabla (u_{2,\fine} - u_{2,\coarse})}{\Omega}^2;
\end{align*}
see, e.g., \cite{ckns08} for details. 
For the second term, we use the fact that that $\nabla_\Gamma(u_{1,\fine} - u_{1,\coarse} - c)$ is a polynomial of degree $p-1$ on $T\in\TT_\coarse\cap\TT_\fine$, stability of $\Pi^\Gamma_\coarse$, the inverse inequality~\cite[Corollary~3.2]{affkmp17}, the trace inequality, and the Poincar\'e inequality to see that 
\begin{align*}
	&\sum_{T\in\TT_\coarse\cap\TT_\fine} h_T \norm{(1-\Pi^\Gamma_\coarse)\nabla_\Gamma (K-1/2) (u_{1,\fine}-u_{1,\coarse})}{\partial T\cap \Gamma}^2
	\\
	&\qquad =
	\sum_{T\in\TT_\coarse\cap\TT_\fine} h_T \norm{(1-\Pi^\Gamma_\coarse)\nabla_\Gamma K (u_{1,\fine}-u_{1,\coarse} - c )}{\partial T\cap \Gamma}^2
	\\
	&\qquad \lesssim \norm{u_{1,\fine}-u_{1,\coarse} -c }{H^{1/2}(\Gamma)}^2 
	\lesssim \norm{\nabla(u_{1,\fine}-u_{1,\coarse})}{\Omega}^2.
\end{align*}
Overall, this concludes the proof. \hfill$\square$

\subsection{Reduction on refined elements} \label{sec:reduction}
We show reduction on nonrefined elements, i.e., there exist a generic contraction constant $0<\rho_{\rm red}<1$ and a uniform constant $\const{stab}>0$ 
depending only on the domain $\Omega$, the shape-regularity of the considered triangulations (and thus on $\TT_0$), and the used polynomial degree $p$ such that for all refinements $\TT_\fine\in\T(\TT_\coarse)$ of a conforming triangulation $\TT_\coarse\in\T$, it holds that 
\begin{align}
	\sum_{T\in\TT_\fine\setminus\TT_\coarse} \eta_\fine(T)^2 
	\le 
	\rho_{\rm red}^2 \sum_{T\in\TT_\coarse\setminus\TT_\fine} \eta_\coarse(T)^2 
	+ \const{red}^2 \norm{\nabla (\u_\fine - \u_\coarse) }{\Omega}^2.
\end{align}

Indeed, for the component $u_1$, this follows from standard arguments as in Section~\ref{sec:stability}, i.e.,
\begin{align}
	\sum_{T\in\TT_\fine\setminus\TT_\coarse} \eta_{1,\fine}(T)^2 
	\le 
	\rho_{\rm 1, red}^2 \sum_{T\in\TT_\coarse\setminus\TT_\fine} \eta_{1,\coarse}(T)^2 
	+ \const{1,red}^2 \norm{\nabla (u_{1,\fine} - u_{1,\coarse}) }{\Omega}^2
\end{align}
with constants $0<\rho_{\rm 1, red}<1$ and $\const{1,red} > 0$. 

Standard arguments also show for the volume residual and the jump terms in the definition of the estimator corresponding to the second component~\eqref{eq:estimator2} that
\begin{align*}
	&\sum_{T\in\TT_\fine\setminus\TT_\coarse}  h_T^2 \norm{\Delta u_{2,\fine}}{T}^2 + h_T \norm{ [\![ \partial_{\n} u_{2,\fine}]\!] }{\partial T\cap \Omega}^2 
	\\ 
	&\quad \le \tilde{\rho}_{\rm 2,red}^2 \Big(\sum_{T\in\TT_\coarse\setminus\TT_\fine} h_T^2 \norm{\Delta (u_{2,\fine} - u_{2,\coarse})}{T}^2
	+  h_T \norm{ [\![ \partial_{\n} u_{2,\coarse}]\!] }{\partial T\cap \Omega}^2 \Big)
	+ \tilde{C}_{\rm red,2}^2 \norm{\nabla (u_{2,\fine} - u_{2,\coarse}) }{\Omega}^2
\end{align*}
with constants $0<\tilde{\rho}_{\rm 2,red}<1$ and $\tilde{C}_{\rm red}>0$.
The triangle and the Young inequality, the inclusion $\PP^{p-1}(\TT_\coarse|_\Gamma)\subseteq \PP^{p-1}(\TT_\fine|_\Gamma)$, and the fact that $\bigcup(\TT_\fine \setminus \TT_\coarse) = \bigcup(\TT_\coarse \setminus \TT_\fine)$, where $h_T \le 2^{-1/d} h_{T'}$ for all $T'\in \TT_\coarse \setminus \TT_\fine$ and $T\subseteq T'$ with $T\in \TT_\fine \setminus \TT_\coarse$, show that 
\begin{align*}
	&\sum_{T\in\TT_\fine\setminus\TT_\coarse} h_T \norm{(1-\Pi^\Gamma_\fine)\nabla_\Gamma (K-1/2) (u_{1,\fine} - g)}{\partial T\cap \Gamma}^2 
	\\
	&\quad\le 
	(1+\delta) \sum_{T\in\TT_\fine\setminus\TT_\coarse} h_T \norm{(1-\Pi^\Gamma_\coarse)\nabla_\Gamma (K-1/2) (u_{1,\coarse} - g)}{\partial T\cap \Gamma}^2 
	\\
	&\quad\quad + (1+\delta^{-1}) \sum_{T\in\TT_\fine\setminus\TT_\coarse}  h_T \norm{(1-\Pi^\Gamma_\fine)\nabla_\Gamma (K-1/2) (u_{1,\fine}-u_{1,\coarse})}{\partial T\cap \Gamma}^2
	\\
	&\quad\le 
	(1+\delta) 2^{-1/d} \sum_{T'\in\TT_\coarse\setminus\TT_\fine} h_{T'} \norm{(1-\Pi^\Gamma_\coarse)\nabla_\Gamma (K-1/2) (u_{1,\coarse} - g)}{\partial T'\cap \Gamma}^2 
	\\
	&\quad\quad + (1+\delta^{-1}) \sum_{T\in\TT_\fine\setminus\TT_\coarse} h_{T} \norm{(1-\Pi^\Gamma_\fine)\nabla_\Gamma (K-1/2) (u_{1,\fine}-u_{1,\coarse})}{\partial T\cap \Gamma}^2.
\end{align*}
The second sum can be bounded as in Section~\ref{sec:stability} by $\norm{\nabla (u_{1,\fine}-u_{1,\coarse})}{\Omega}^2$, and we conclude the proof by choosing $\delta>0$ sufficiently small. \hfill$\square$

\subsection{Discrete reliability}\label{sec:discrete_reliability}

We show discrete reliability, i.e., there exist uniform constants $\const{drel}, \const{ref}>0$ depending only the domain $\Omega$, the initial triangulation $\TT_0$ (or more precisely on the shape-regularity of the considered triangulations and the patch shapes on the boundary, see~Proposition~\ref{prop:reliability}), and the used polynomial degree $p$ such that for all refinements $\TT_\fine\in\T(\TT_\coarse)$ of a conforming triangulation $\TT_\coarse\in\T$, there exists a set $\TT_\coarse \setminus \TT_\fine \subseteq \mathcal{R}_{\coarse,\fine} \subseteq \TT_\coarse$ with $\#\RR_{\coarse,\fine} \le \const{ref} \#(\TT_\coarse \setminus \TT_\fine)$ and
\begin{align}\label{eq:discrete_reliability}
	\norm{\nabla(\u_\fine - \u_\coarse)}{\Omega}
	\le \norm{\u_\fine - \u_\coarse}{H^1(\Omega)}
	\le \const{drel} \Big(\sum_{T\in\RR_{\coarse,\fine}} \eta_\coarse(T)^2 \Big)^{1/2}.
\end{align}

Standard arguments show the assertion for the first component $u_1$, i.e., 
\begin{equation} \label{eq:discrete_reliability_1}
	\norm{u_{1,\fine} - u_{1,\coarse}}{H^1(\Omega)}
	\lesssim \Big(\sum_{T\in\RR_{1,\coarse,\fine}} \eta_{1,\coarse}(T)^2\Big)^{1/2}
\end{equation}
for some $\TT_\coarse \setminus \TT_\fine \subseteq \mathcal{R}_{1,\coarse,\fine} \subseteq \TT_\coarse$ with $\#\RR_{1,\coarse,\fine} \lesssim \#(\TT_\coarse \setminus \TT_\fine)$; see, e.g., \cite{ckns08}. 

For the second component $u_2$, we require a discrete version of the harmonic lifting operator $L:H^{1/2}(\Gamma) \to H^1(\Omega)$ from~\eqref{eq:lifting}: 
Given $w\in H^{1/2}(\Gamma)$, let $L_\fine w \in \SS^p(\TT_\fine)$ with $(L_\fine w)|_\Gamma = J_\fine^\Gamma w$ such that 
\begin{equation}\label{eq:discrete_lifting}
	\inner[\Omega]{\nabla L_\fine w}{\nabla v_\fine}
	= 0
	\quad \text{for all } v_\fine \in \SS^p_0(\TT_\fine).
\end{equation}
From $(J_\fine^\Omega L (\cdot))|_\Gamma = J_\fine^\Gamma (\cdot)$ (see~\eqref{eq:sz_restricted}), we conclude that $(L_\fine - J_\fine^\Omega L J_\fine^\Gamma) w\in \SS^p_0(\TT_\fine)$ is the Galerkin approximation of $(L J_\fine^\Gamma - J_\fine^\Omega L J_\fine^\Gamma) w \in H^1_0(\Omega)$. 
The triangle inequality and stability of the Galerkin approximation (together with Friedrichs' inequality), of $L J_\fine^\Gamma$, and of $J_\fine^\Omega L J_\fine^\Gamma$ imply that
\begin{align*}
	\norm{L_\fine w}{H^1(\Omega)} 
	&\le \norm{(L_\fine - J_\fine^\Omega L J_\fine^\Gamma) w}{H^1(\Omega)} + \norm{J_\fine^\Omega L J_\fine^\Gamma w}{H^1(\Omega)}
	\\
	&\lesssim \norm{L J_\fine^\Gamma w}{H^1(\Omega)} + \norm{J_\fine^\Omega L J_\fine^\Gamma w}{H^1(\Omega)}
	\lesssim \norm{w}{H^{1/2}(\Gamma)},
\end{align*}
i.e., $L_\fine: H^{1/2}(\Omega) \to H^1(\Omega)$ is uniformly bounded. 
In particular, we have by definition that $u_{2,\fine} = L_\fine(K-1/2)(u_{1,\fine} - g)$. 
We introduce $u_{2,\fine}^\coarse := L_\fine(K-1/2) (u_{1,\coarse} - g)$ and apply the triangle inequality to see that
\begin{equation*}
	\norm{u_{2,\fine} - u_{2,\coarse}}{H^1(\Omega)} 
	\le \norm{u_{2,\fine} - u_{2,\fine}^\coarse}{H^1(\Omega)} + \norm{u_{2,\fine}^\coarse - u_{2,\coarse}}{H^1(\Omega)}.
\end{equation*}
Noting that $u_{2,\fine} = L_\fine(K-1/2)(u_{1,\fine} - g)$, uniform stability of $L_\fine$, of $K$, and of the trace operator yield together with discrete reliability~\eqref{eq:discrete_reliability_1} for $u_1$ that 
\begin{align*}
	\norm{u_{2,\fine} - u_{2,\fine}^\coarse}{H^1(\Omega)}
	&\lesssim \norm{(K-1/2)(u_{1,\fine} - u_{1,\coarse}) }{H^{1/2}(\Gamma)}
	\\
	&\lesssim \norm{u_{1,\fine} - u_{1,\coarse} }{H^{1/2}(\Gamma)}
	\lesssim \norm{u_{1,\fine} - u_{1,\coarse} }{H^{1}(\Omega)}
	\\
	&\lesssim \Big(\sum_{T\in\RR_{1,\coarse,\fine}} \eta_{1,\coarse}(T)\Big)^{1/2}.
\end{align*}
Moreover, the discrete reliability of \cite[Proposition~6.1]{afkpp13} states that 
\begin{align*}
	\norm{u_{2,\fine}^\coarse - u_{2,\coarse}}{H^1(\Omega)}
	\lesssim \Big(\sum_{T\in\RR_{2,\coarse,\fine}} \eta_{2,\coarse}(T)\Big)^{1/2}
\end{align*}
for some $\TT_\coarse \setminus \TT_\fine \subseteq \mathcal{R}_{2,\coarse,\fine} \subseteq \TT_\coarse$ with $\#\RR_{2,\coarse,\fine} \lesssim \#(\TT_\coarse \setminus \TT_\fine)$.
Taking $\RR_{\coarse,\fine} := \RR_{1,\coarse,\fine} \cup \RR_{2,\coarse,\fine}$, we conclude the proof. \hfill$\square$

\subsection{General quasi-orthogonality}\label{sec:orthogonality}

We show in two steps general quasi-orthogonality, i.e., for all $\varepsilon_{\rm qo}>0$, there exists $\const{qo} = \const{qo}(\varepsilon_{\rm qo}) > 0$ 
depending only the domain $\Omega$, the initial triangulation $\TT_0$, and the used polynomial degree $p$ such that 
\begin{align}\label{eq:general_quasi_orthogonality}
	\sum_{j = \ell}^{\ell + N} \norm{\nabla(\u_{j+1} - \u_{j} )}{\Omega}^2 - \varepsilon_{\rm qo} \eta_j^2
	\le 
	\const{qo} \eta_\ell^2
	\quad \text{for all }\ell,N \in \N_0.
\end{align}

{\bf Step~1:} In this step, we prove that for arbitrary fixed $\varepsilon_{\rm qo}>0$, there exist $\varepsilon_1,\varepsilon_2>0$ with $\varepsilon_1 \const{rel}^2 + \varepsilon_2 \le \varepsilon_{\rm qo}$ such that for all $j\in\N_0$, 
\begin{subequations}\label{eq:quasi_orthogonality}
\begin{align}\label{eq:quasi_orthogonality_main}
\begin{split}
	\norm{\nabla (\u_{j+1} - \u_j) }{\Omega}^2 
	\le 
	\norm{\nabla (\u - \u_j) }{\Omega}^2  - (1-\varepsilon_1) \norm{\nabla (\u - \u_{j+1}) }{\Omega}^2
	+ \varepsilon_2 \eta_j^2
	+ \alpha_j^2 - \alpha_{j+1}^2.
\end{split}
\end{align}
For all $j \in \N_0$, the auxiliary terms $\alpha_j$ are nonnegative numbers possibly depending on $\varepsilon_1$ and $\varepsilon_2$ with 
\begin{align}\label{eq:bound_osc}
	\alpha_j \le C_\alpha(\varepsilon_1,\varepsilon_2) \eta_j
\end{align}
\end{subequations}
for some constant $C_\alpha(\varepsilon_1,\varepsilon_2)>0$. 
We split this step into six substeps. 

{\bf Step~1.1:}
For the first component $u_1$, Galerkin orthogonality even yields the Pythagoras identity 
\begin{align*}
	\norm{\nabla (u_{1,j+1} - u_{1,j}) }{\Omega}^2 
	=
	\norm{\nabla (u_1 - u_{1,j}) }{\Omega}^2 - \norm{\nabla (u_1 - u_{1,j+1}) }{\Omega}^2.
\end{align*}

{\bf Step~1.2:}
For the second component $u_2$, we start with the following elementary identity
\begin{align*}
	&\norm{\nabla (u_{2} - u_{2,j}) }{\Omega}^2 
	= 
	\norm{\nabla (u_{2} - u_{2,j+1}  +  u_{2,j+1} - u_{2,j}) }{\Omega}^2 
	\\
	&\quad=
	\norm{\nabla (u_2 - u_{2,j+1}) }{\Omega}^2 + 2 \inner[\Omega]{\nabla (u_2 - u_{2,j+1}) }{\nabla (u_{2,j+1} - u_{2,j}) } + \norm{\nabla (u_{2,j+1} - u_{2,j}) }{\Omega}^2.
\end{align*}
Rearranging this identity, we see  that 
\begin{align}
	\notag
	\norm{\nabla (u_{2,j+1} - u_{2,j}) }{\Omega}^2
	=
	\norm{\nabla (u_{2} - u_{2,j}) }{\Omega}^2 
	- \norm{\nabla (u_2 - u_{2,j+1}) }{\Omega}^2
	\\
	\label{eq:qo_scalar_product}
	- 2 \inner[\Omega]{\nabla (u_2 - u_{2,j+1}) }{\nabla (u_{2,j+1} - u_{2,j}) }.
\end{align}

{\bf Step~1.3:}
It remains to estimate the last term on the right-hand side of \eqref{eq:qo_scalar_product}. 
Recalling the discrete lifting operator from~\eqref{eq:discrete_lifting}, we set  $u_{2,j+1}^j:= L_{j+1} J_{j}^\Gamma(K-1/2)(u_{1,j} - g)\in \SS^p(\TT_{j+1})$ with $u_{2,j+1}^j|_\Gamma = J_{j+1}^\Gamma J_{j}^\Gamma(K-1/2)(u_{1,j} - g) = J_{j}^\Gamma(K-1/2)(u_{1,j} - g)$. 
Then, the difference $u_{2,j+1}^j - u_{2,j}$ lies in $\SS^p_0(\TT_{j+1})$. 
Moreover, we note that $u_{2,j+1} = L_{j+1} (K-1/2)(u_{1,j+1} - g) = L_{j+1} J_{j+1}^\Gamma(K-1/2)(u_{1,j+1} - g)$. 
For arbitrary $\varepsilon_1>0$, Galerkin orthogonality, the Young inequality, and uniform stability of $L_{j+1}$ hence imply for the $L^2$-scalar product~\eqref{eq:qo_scalar_product} that 
\begin{align*}
	&2 |\inner[\Omega]{\nabla (u_2 - u_{2,j+1}) }{\nabla (u_{2,j+1} - u_{2,j}) }|
	= 2 |\inner[\Omega]{\nabla (u_2 - u_{2,j+1}) }{\nabla (u_{2,j+1} - u_{2,j+1}^j) }|
	\\
	&\le \varepsilon_1 \norm{\nabla (u_{2} - u_{2,j+1}) }{\Omega}^2 + \varepsilon_1^{-1} \norm{\nabla (u_{2,j+1} - u_{2,j+1}^j) }{\Omega}^2
	\\
	&\le \varepsilon_1 \norm{\nabla (u_{2} - u_{2,j+1}) }{\Omega}^2 + \varepsilon_1^{-1} C_1 \norm{J_{j+1}^\Gamma(K-1/2)(u_{1,j+1} - g) - J_{j}^\Gamma(K-1/2)(u_{1,j} - g)}{H^{1/2}(\Gamma)}^2
\end{align*}
for some uniform constant $C_1>0$.

{\bf Step~1.4:}
With the triangle and the Young inequalities, stability of the involved operators, and the fact that $u_{1,j}$ and $u_{1,j+1}$ have integral mean zero, we further see for some uniform $C_2>0$ that
\begin{align*}
	&\norm{J_{j+1}^\Gamma(K-1/2)(u_{1,j+1} - g) - J_{j}^\Gamma(K-1/2)(u_{1,j} - g)}{H^{1/2}(\Gamma)}^2
	\\
	&\quad\le 
	2\norm{J_{j+1}^\Gamma(K-1/2)(u_{1,j+1} - g) - J_{j+1}^\Gamma(K-1/2)(u_{1,j} - g)}{H^{1/2}(\Gamma)}^2
	\\
	&\qquad+ 2 \norm{J_{j+1}^\Gamma(K-1/2)(u_{1,j} - g) - J_{j}^\Gamma(K-1/2)(u_{1,j} - g)}{H^{1/2}(\Gamma)}^2
	\\
	&\quad=
	2 \norm{J_{j+1}^\Gamma(K-1/2)(u_{1,j} - u_{1,j+1})}{H^{1/2}(\Gamma)}^2 + 2 \norm{(J_{j+1}^\Gamma -J_j^\Gamma) (K-1/2)(u_{1,j} - g)}{H^{1/2}(\Gamma)}^2
	\\
	&\quad\le 
	2 C_2 \norm{\nabla(u_{1,j+1} - u_{1,j}) }{\Omega}^2 + 2 \norm{(J_{j+1}^\Gamma -J_j^\Gamma) (K-1/2)(u_{1,j} - g)}{H^{1/2}(\Gamma)}^2.
\end{align*}

{\bf Step~1.5:}
For an equivalent $\TT_\coarse|_\Gamma$-piecewise constant mesh-size function $\tilde{h}_\coarse$ on $\Gamma$ satisfying $C_3^{-1} h_T \le \tilde{h}_\coarse|_{\partial T\cap \Gamma} \le C_3 h_T$  for some uniform $C_3>0$ and all $T\in\TT_\coarse\in\T$ with $\partial T\cap \Gamma \neq \emptyset$, the proof of \cite[Proposition~11.1]{cfpp14} reveals the existence of a uniform constant $C_4>0$ such that 
\begin{align*}
	\norm{(J_{j+1}^\Gamma -J_j^\Gamma) (K-1/2)(u_{1,j} - g)}{H^{1/2}(\Gamma)}^2
	&\le 
	C_4 \big[\norm{\tilde{h}_j^{1/2} (1-\Pi^\Gamma_j) \nabla_\Gamma (K-1/2)(u_{1,j} - g)}{\Gamma}^2 
	\\
	&\quad - \norm{\tilde{h}_{j+1}^{1/2} (1-\Pi^\Gamma_{j+1}) \nabla_\Gamma (K-1/2)(u_{1,j} - g)}{\Gamma}^2\big].
\end{align*}
For the second term, we employ the Young inequality $-a^2 \le -1/(1+\delta) (a+b)^2 + (1+\delta^{-1})/(1+\delta) b^2$ for $a,b\ge0, \delta>0$
in combination with the triangle inequality, local $L^2$-stability of $\Pi^\Gamma_{j+1}$, and the fact that $(1-\Pi^\Gamma_{j+1}) \nabla_\Gamma (u_{1,j+1}-u_{1,j})=0$   to see that 
\begin{align*}
	&- \norm{\tilde{h}_{j+1}^{1/2} (1-\Pi^\Gamma_{j+1}) \nabla_\Gamma (K-1/2)(u_{1,j} - g)}{\Gamma}^2 
	\\
	&\quad\le 
	- \frac{1}{1+\delta} \norm{\tilde{h}_{j+1}^{1/2} (1-\Pi^\Gamma_{j+1}) \nabla_\Gamma (K-1/2)(u_{1,j+1} - g)}{\Gamma}^2 
	\\
	&\quad\quad+ \frac{1+\delta^{-1}}{1+\delta} \norm{\tilde{h}_{j+1}^{1/2} (1-\Pi^\Gamma_{j+1}) \nabla_\Gamma (K-1/2)(u_{1,j+1} - u_{1,j} )}{\Gamma}^2 
	\\
	&\quad\le 
	- \frac{1}{1+\delta} \norm{\tilde{h}_{j+1}^{1/2} (1-\Pi^\Gamma_{j+1}) \nabla_\Gamma (K-1/2)(u_{1,j+1} - g)}{\Gamma}^2 
	\\
	&\quad\quad+ \frac{1+\delta^{-1}}{1+\delta} \norm{\tilde{h}_{j+1}^{1/2} \nabla_\Gamma K(u_{1,j+1} - u_{1,j} )}{\Gamma}^2.
\end{align*}
Finally, an application of the inverse estimate from \cite[Corollary~3.2]{affkmp17} in combination with the stability of the trace operator and the fact that $u_{1,j+1} - u_{1,j}$ has integral mean zero yields the existence of uniform constants $C,C_5>0$ with  
\begin{align*}
	\norm{\tilde{h}_{j+1}^{1/2} \nabla_\Gamma K(u_{1,j+1} - u_{1,j} )}{\Gamma}^2 
	\le C \norm{u_{1,j+1} - u_{1,j}}{H^{1/2}(\Gamma)}^2 
	\le C_5 \norm{\nabla(u_{1,j+1} - u_{1,j} )}{\Omega}^2.
\end{align*}

{\bf Step~1.6:}
Combining Step~1.1--1.5, we see that 
\begin{align*}
	&\norm{\nabla (u_{2,j+1} - u_{2,j}) }{\Omega}^2
	\le 
	\norm{\nabla (u_{2} - u_{2,j}) }{\Omega}^2 - (1-\varepsilon_1) \norm{\nabla (u_2 - u_{2,j+1}) }{\Omega}^2
	\\
	& + 2\varepsilon_1^{-1} C_1 \Big\{ \Big(C_2 + C_4 C_5\frac{1+ \delta^{-1}}{1+\delta}\Big) \big[ \norm{\nabla(u_1 - u_{1,j}) }{\Omega}^2 - \norm{\nabla(u_1 - u_{1,j+1}) }{\Omega}^2 \big]
	\\
	&+ C_4 \big[ \norm{\tilde{h}_j^{1/2} (1-\Pi^\Gamma_j) \nabla_\Gamma (K-1/2)(u_{1,j} - g)}{\Gamma}^2 \\
	& - \frac{1}{1+\delta} \norm{\tilde{h}_{j+1}^{1/2} (1-\Pi^\Gamma_{j+1}) \nabla_\Gamma (K-1/2)(u_{1,j+1} - g)}{\Gamma}^2 \big]\Big\}.
\end{align*}
Recall that $\norm{\tilde{h}_j^{1/2} (1-\Pi^\Gamma_j) \nabla_\Gamma (K-1/2)(u_{1,j} - g)}{\Gamma}^2 \le C_3 \eta_{2,j}^2$ by definition of the estimator~\eqref{eq:estimator2}. 
With the auxiliary terms 
\begin{align*}
	\alpha_j^2 := 2\varepsilon_1^{-1} C_1 &\Big\{\Big(C_2 + C_4 C_5\frac{1+ \delta^{-1}}{1+\delta}\Big) \norm{\nabla(u_1 - u_{1,j}) }{\Omega}^2 
	\\
	&+ \frac{C_4}{1+\delta} \norm{\tilde{h}_j^{1/2} (1-\Pi^\Gamma_j) \nabla_\Gamma (K-1/2)(u_{1,j} - g)}{\Gamma}^2\Big\}
\end{align*}
for all $j\in\N_0$, we hence arrive at
\begin{align*}
	\norm{\nabla (u_{2,j+1} - u_{2,j}) }{\Omega}^2
	&\le 
	\norm{\nabla (u_{2} - u_{2,j}) }{\Omega}^2 - (1-\varepsilon_1) \norm{\nabla (u_2 - u_{2,j+1}) }{\Omega}^2
	\\
	&\quad 
	+ 2\varepsilon_1^{-1} C_1 C_4 \Big(1-\frac{1}{1+\delta}\Big) C_3 \eta_{2,j}^2
	+ \alpha_j^2 - \alpha_{j+1}^2.
\end{align*}
Choosing first $\varepsilon_1>0$ such that $\varepsilon_1 \const{rel}^2 < \varepsilon_{\rm qo}$ and then $\delta>0$ such that $\varepsilon_1 \const{rel}^2 + \varepsilon_2 \le \varepsilon_{\rm qo}$ with $\varepsilon_2 := 2\varepsilon_1^{-1} C_1 C_3 C_4 \delta/(1+\delta)$, we conclude the proof of \eqref{eq:quasi_orthogonality_main}. 
The bound~\eqref{eq:bound_osc} follows directly from reliability~\eqref{eq:reliability_efficiency} of the estimator.

{\bf Step~2:} In this step, we show that the quasi-orthogonality~\eqref{eq:quasi_orthogonality} indeed implies general quasi-orthogonality~\eqref{eq:general_quasi_orthogonality}. 
Let $\varepsilon_{\rm qo} > 0$ be arbitrary but fixed. 
We choose $\varepsilon_1,\varepsilon_2>0$ with $\varepsilon_1 \const{rel}^2 + \varepsilon_2 \le \varepsilon_{\rm qo}$ as in Step~1.
Then, it holds that
\begin{align*}
	\varepsilon_1 \, \norm{\nabla (\u - \u_j)}{\Omega}^2 + (\varepsilon_2 - \varepsilon_{\rm qo}) \eta_j^2 \le 0
	\quad \text{for all } j \in \N_0.
\end{align*}
Together with~\eqref{eq:quasi_orthogonality}, this shows for all $\ell,N \in \N_0$ that
\begin{align*}
	&\sum_{j = \ell}^{\ell+N}  \norm{\nabla (\u_{j+1} - \u_j) }{\Omega}^2 - \varepsilon_{\rm qo} \eta_j^2 
	\le 
	\sum_{j=\ell}^{\ell+N} \norm{\nabla (\u - \u_j) }{\Omega}^2  - (1-\varepsilon_1) \norm{\nabla (\u - \u_{j+1}) }{\Omega}^2
	\\
	&\hspace{6cm}+ (\varepsilon_2 - \varepsilon_{\rm qo}) \eta_j^2
	+ \alpha_j^2 - \alpha_{j+1}^2 
	\\
	&\quad\le 
	\sum_{j=\ell}^{\ell+N} (1-\varepsilon_1) \norm{\nabla (\u - \u_j) }{\Omega}^2  - (1-\varepsilon_1) \norm{\nabla (\u - \u_{j+1}) }{\Omega}^2 + \alpha_j^2 - \alpha_{j+1}^2 
	\\
	&\quad= (1-\varepsilon_1) \norm{\nabla (\u - \u_\ell) }{\Omega}^2  - (1-\varepsilon_1) \norm{\nabla (\u - \u_{\ell+N+1}) }{\Omega}^2  + \alpha_\ell^2 - \alpha_{\ell+N+1}^2 
	\\
	&\quad\le (1-\varepsilon_1) \norm{\nabla (\u - \u_\ell) }{\Omega}^2  + \alpha_\ell^2.
\end{align*}
With reliability and the bound~\eqref{eq:bound_osc}, we conclude the proof.\hfill$\square$

\section{Numerical experiments}\label{sec:numerics}

In this section, we present some numerical experiments
in two dimensions ($d=2$) for first-order finite elements ($p=1$).
The aim of our computations is twofold:
On the one hand, we aim to illustrate the reliability and efficiency
of the hybrid FEM-BEM method~\eqref{eq:u1_coarse}--\eqref{eq:u2_coarse} to solve
the full-space linear elliptic transmission problem~\eqref{eq:transmission}.
On the other hand, for the case of singular solutions,
we aim to show the superiority of adaptive mesh refinement
(steered by Algorithm~\ref{alg:adaptive_algorithm})
over uniform mesh refinement.

The numerical results presented in this section were obtained with
an implementation based on the MATLAB libraries
\texttt{p1afem}~\cite{fpw11}
(assembly of the finite element matrices, computation of the error estimates, mesh management and refinement)
and
\texttt{HILBERT}~\cite{aefffgkmp14}
(evaluation of the double-layer operator necessary to realize the Dirichlet condition in~\eqref{eq:u2_coarse}).
The discrete equations are assembled in the MATLAB sparse format and solved with the MATLAB backslash
operator.

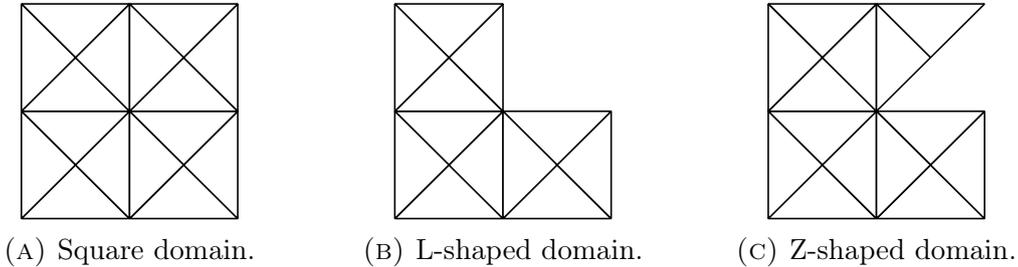
\begin{figure}[ht]
\centering
\begin{subfigure}[b]{0.3\textwidth}
\centering
\begin{tikzpicture}[scale = 0.6]
\begin{axis}[
axis lines = none,
axis equal,
]
\addplot[patch,
color=white, 
faceted color = black,
line width = 1pt,
patch table = {./pics/square/elements.dat}] file{./pics/square/coordinates.dat};
\end{axis}
\end{tikzpicture}
\caption{Square domain.}
\label{fig:initial_mesh_square}
\end{subfigure}
\begin{subfigure}[b]{0.3\textwidth}
\centering
\begin{tikzpicture}[scale = 0.6]
\begin{axis}[
axis lines = none,
axis equal,
]
\addplot[patch,
color=white, 
faceted color = black,
line width = 1pt,
patch table = {./pics/Lshaped/elements.dat}] file{./pics/Lshaped/coordinates.dat};
\end{axis}
\end{tikzpicture}
\caption{L-shaped domain.}
\label{fig:initial_mesh_lshaped}
\end{subfigure}
\begin{subfigure}[b]{0.3\textwidth}
\centering
\begin{tikzpicture}[scale = 0.6]
\begin{axis}[
axis lines = none,
axis equal,
]
\addplot[patch,
color=white, 
faceted color = black,
line width = 1pt,
patch table = {./pics/Zshaped/elements.dat}] file{./pics/Zshaped/coordinates.dat};
\end{axis}
\end{tikzpicture}
\caption{Z-shaped domain.}
\label{fig:initial_mesh_zshaped}
\end{subfigure}
\caption{Domains and initial meshes $\TT_0$.}
\label{fig:initial_meshes}
\end{figure}

In the following three sections,
we present numerical results
for $\Omega \subset \R^2$ being
either the square domain $(-1/4,1/4)^2$,
or the L-shaped domain $(-1/4,1/4)^2\setminus[0,1/4)^2$,
or the Z-shaped domain $(-1/4,1/4)^2\setminus\mathrm{conv}\{ (0,0) , (1/4,0) , (1/4,1/4) \}$.
The domains and the respective meshes $\TT_0$ used to initialize Algorithm~\ref{alg:adaptive_algorithm}
are depicted in Figure~\ref{fig:initial_meshes}.

\subsection{Square domain} \label{sec:numerics_square}

We consider problem~\eqref{eq:transmission} posed on the square domain $\Omega = (-1/4,1/4)^2$.
The data $f$, $g$, and $\phi$ are chosen in such a way that the exact solution is given by
\begin{equation*}
u(x_1,x_2) = \cos(2\pi x_1) \cos(2\pi x_2),
\quad
\uext(x_1,x_2) = \frac{x_1 + x_2}{x_1^2 + x_2^2},
\end{equation*}
and they satisfy the compatibility condition~\eqref{eq:compatibility}.
Note that the solution $u$ is smooth in $\Omega$.

Starting from an initial mesh $\TT_0$ made of 16 elements (see Figure~\ref{fig:initial_mesh_square}),
we perform 9 steps of uniform refinement,
where in each step each triangle is split into four elements by three bisections (a so-called bisec(3)-operation).
For each mesh $\TT_\ell$, we consider the corresponding approximation $u_\ell$ generated by
the hybrid FEM-BEM method~\eqref{eq:u1_coarse}--\eqref{eq:u2_coarse}
($\ell = 0, \dots, 9$).
Note that the same sequence of approximations
can be obtained by running 10 steps of Algorithm~\ref{alg:adaptive_algorithm} with $\theta = 1$. 

\begin{figure}[ht]
\centering
\begin{subfigure}[b]{0.31\textwidth}
\centering
\includegraphics[trim = 40 30 45 30, clip, width=\textwidth]{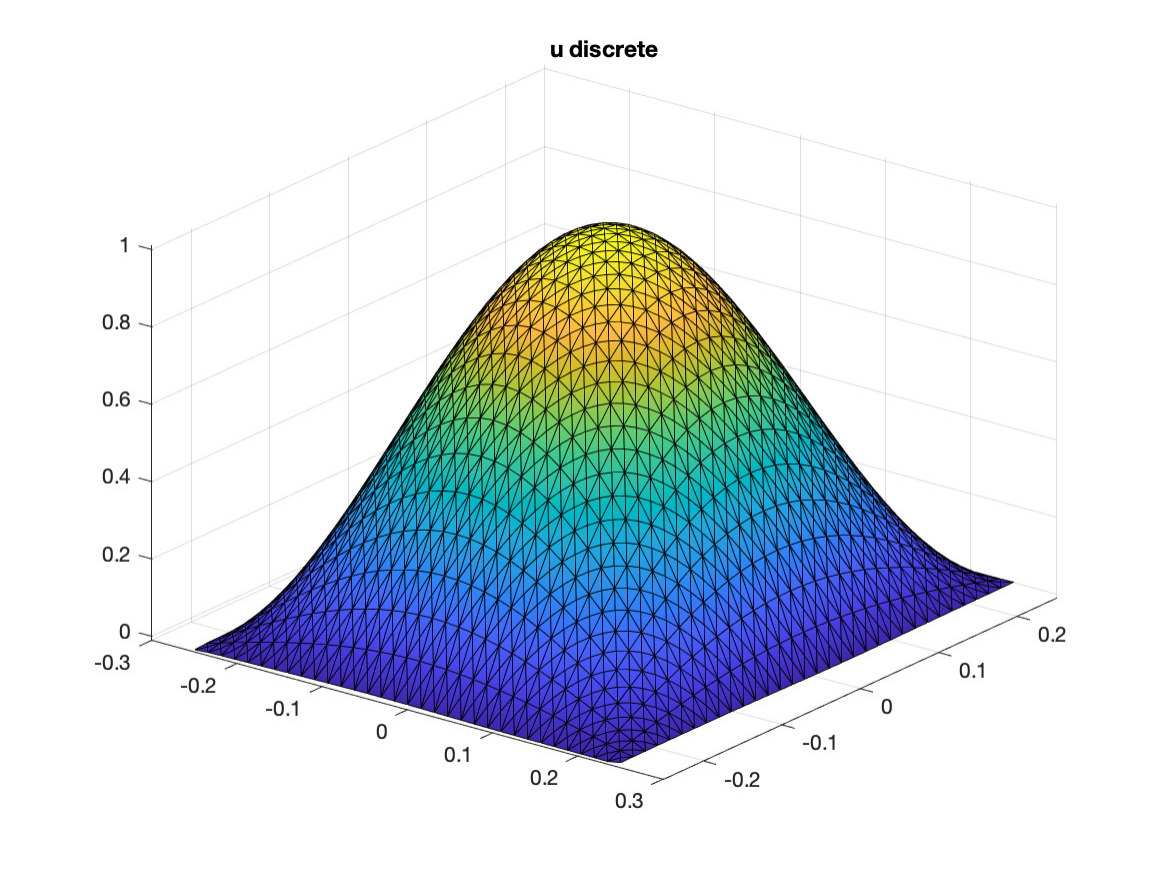}
\caption{$u_4 = u_{1,4} + u_{2,4}$.}
\end{subfigure}
\hfill
\begin{subfigure}[b]{0.31\textwidth}
\centering
\includegraphics[trim = 40 30 45 30, clip, width=\textwidth]{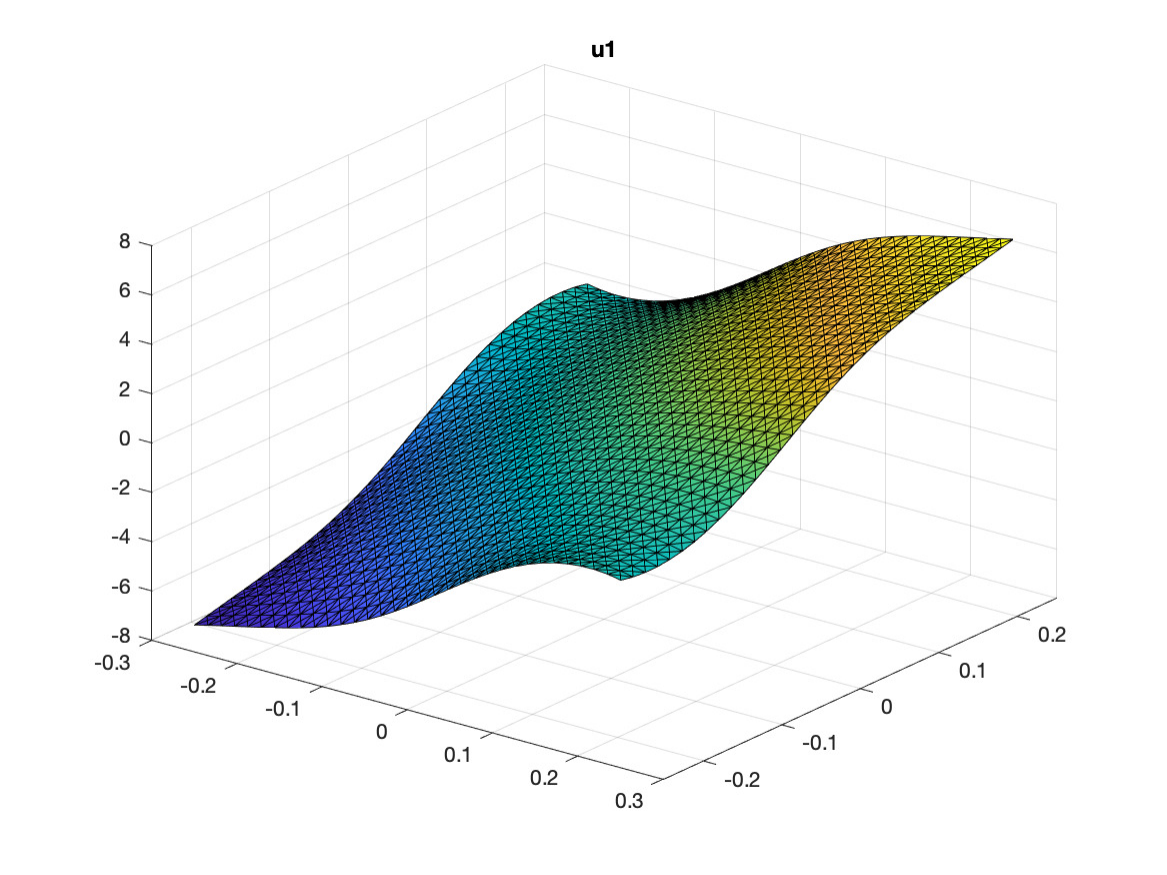}
\caption{$u_{1,4}$.}
\end{subfigure}
\hfill
\begin{subfigure}[b]{0.31\textwidth}
\centering
\includegraphics[trim = 40 30 45 30, clip, width=\textwidth]{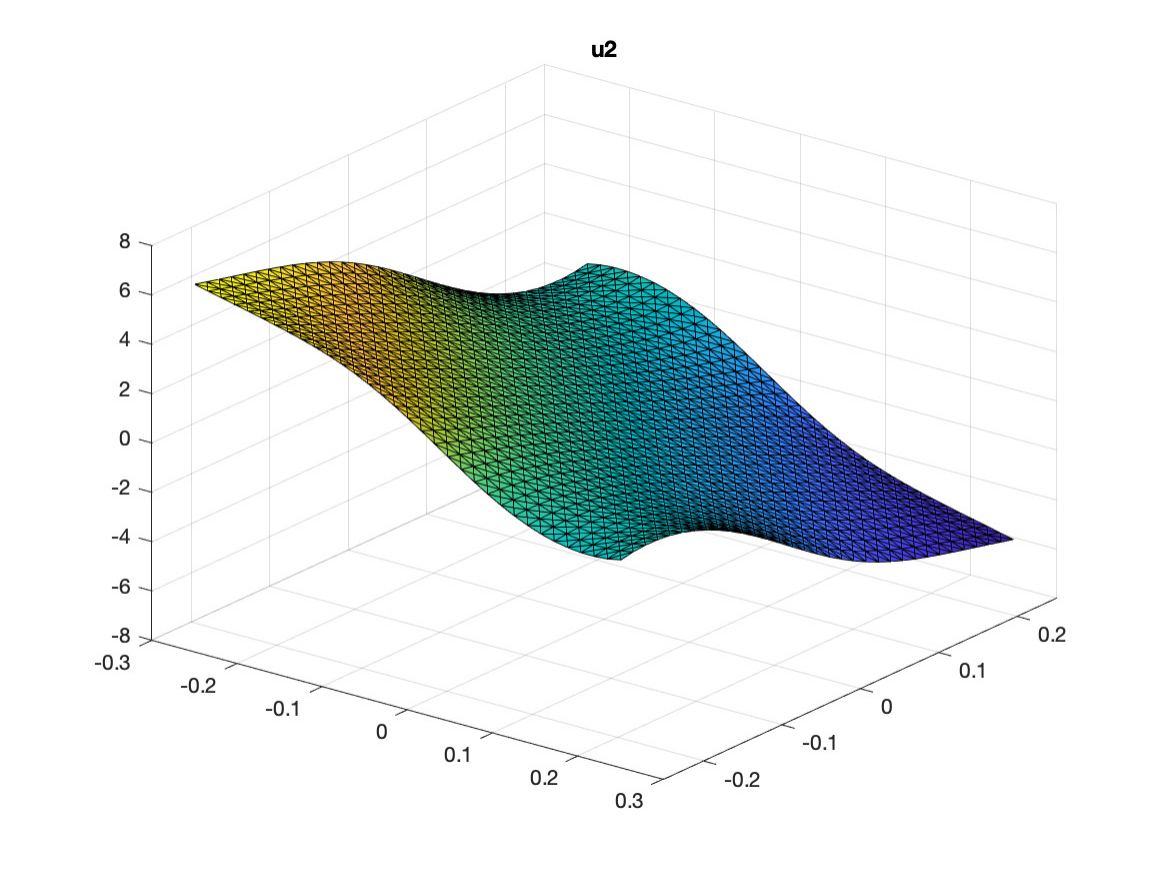}
\caption{$u_{2,4}$.}
\end{subfigure}
\caption{Experiments in Section~\ref{sec:numerics_square}:
Plots of the approximation $u_4$ and of its components $u_{1,4}$ and $u_{2,4}$.}
\label{fig:square_approx}
\end{figure}

In Figure~\ref{fig:square_approx}, to illustrate the working principle of our hybrid FEM-BEM method, 
we show the plot of the fifth approximation $u_4 = u_{1,4} + u_{2,4}$
together with those of its components $u_{1,4}$ (the solution of~\eqref{eq:u1_coarse})
and $u_{2,4}$ (the solution of~\eqref{eq:u2_coarse}).

\begin{figure}[ht]
\centering
\begin{subfigure}[b]{0.48\textwidth}
\centering
\begin{tikzpicture}
\pgfplotstableread{pics/square/results_square_uniform.dat}{\squareuniform}
\begin{loglogaxis}[
width = \textwidth,
xmin = 10,
xmax = 3e6,
ymin=1e-3,
ymax=25,
xlabel={$N_\ell$},
ymajorgrids=true, xmajorgrids=true, grid style=dashed, 
tick label style={font=\tiny},
legend style={
legend pos=north east,
fill=none, draw=none,
font=\tiny
}
]
\addplot[red, mark=diamond, thick] table[x=dofs,y=error]{\squareuniform};
\addplot[blue, mark=+, thick] table[x=dofs,y=est]{\squareuniform};
\addplot[orange, mark=o, thick, densely dash dot] table[x=dofs,y=est1]{\squareuniform};
\addplot[cyan, mark=x, thick, densely dash dot] table[x=dofs,y=est2]{\squareuniform};
\addplot[black,dashed,domain=10^(0):10^(7)] {1.5*x^(-1/2) };
\node at (axis cs:1e1,1.5e-1) [anchor=north west] {\tiny $\mathcal{O}(N_\ell^{-1/2})$};
\legend{
$\norm{u - u_\ell}{H^1(\Omega)}$,
$\eta_\ell$,
$\eta_{1,\ell}$,
$\eta_{2,\ell}$,
}
\end{loglogaxis}
\end{tikzpicture}
\caption{Algorithm~\ref{alg:adaptive_algorithm} with $\theta=1$.}
\label{fig:square_convergence_uniform}
\end{subfigure}
\hfill
\begin{subfigure}[b]{0.48\textwidth}
\centering
\begin{tikzpicture}
\pgfplotstableread{pics/square/results_square_adaptive.dat}{\squareadaptive}
\begin{loglogaxis}[
width = \textwidth,
xmin = 10,
xmax = 3e6,
ymin=1e-3,
ymax=25,
xlabel={$N_\ell$},
ymajorgrids=true, xmajorgrids=true, grid style=dashed, 
tick label style={font=\tiny},
legend style={
legend pos=north east,
fill=none, draw=none,
font=\tiny
}
]
\addplot[red, mark=diamond, thick] table[x=dofs,y=error]{\squareadaptive};
\addplot[blue, mark=+, thick] table[x=dofs,y=est]{\squareadaptive};
\addplot[orange, mark=o, thick, densely dash dot] table[x=dofs,y=est1]{\squareadaptive};
\addplot[cyan, mark=x, thick, densely dash dot] table[x=dofs,y=est2]{\squareadaptive};
\addplot[black,dashed,domain=10^(0):10^(7)] {1.5*x^(-1/2) };
\node at (axis cs:1e1,1.5e-1) [anchor=north west] {\tiny $\mathcal{O}(N_\ell^{-1/2})$};
\legend{
$\norm{u - u_\ell}{H^1(\Omega)}$,
$\eta_\ell$,
$\eta_{1,\ell}$,
$\eta_{2,\ell}$,
}
\end{loglogaxis}
\end{tikzpicture}
\caption{Algorithm~\ref{alg:adaptive_algorithm} with $\theta=1/4$.}
\label{fig:square_convergence_adaptive}
\end{subfigure}
\caption{
Experiments in Section~\ref{sec:numerics_square}:
Error $\norm{u - u_\ell}{H^1(\Omega)}$
and error estimates $\eta_\ell$, $\eta_{1,\ell}$, and $\eta_{2,\ell}$
plotted against the total number of degrees of freedom $N_\ell$.
Comparison of uniform and adaptive mesh refinement.
}
\label{fig:square_convergence}
\end{figure}

In Figure~\ref{fig:square_convergence_uniform},
we plot the approximation error $\norm{u - u_\ell}{H^1(\Omega)}$
and the error estimates $\eta_\ell$, $\eta_{1,\ell}$, and $\eta_{2,\ell}$
against the total number of degrees of freedom $N_\ell$
(i.e., the number of vertices of the mesh).
We see that the error decays with rate $\mathcal{O}(N_\ell^{-1/2})$,
which is the optimal rate for first-order finite elements.
This is in agreement with the \textsl{a~priori} estimate from Proposition~\ref{prop:apriori}.
Here, the best approximation error decays with optimal rate on uniform meshes,
because the solution is regular.
We also see that the total error estimator $\eta_\ell$  and both its components $\eta_{1,\ell}$ and $\eta_{2,\ell}$
decay with the same rate,
which confirms the \textsl{a~posteriori} estimate established in Proposition~\ref{prop:reliability}.

Next, we run Algorithm~\ref{alg:adaptive_algorithm} with $\theta = 1/4$.
In Figure~\ref{fig:square_convergence_adaptive},
we plot the approximation error $\norm{u - u_\ell}{H^1(\Omega)}$
and the error estimates $\eta_\ell$, $\eta_{1,\ell}$, and $\eta_{2,\ell}$
against the total number of degrees of freedom $N_\ell$.
The performance of the adaptive approach is comparable to
the one of uniform mesh refinement.
This is not surprising, given the smoothness of the exact solution.

\subsection{L-shaped domain} \label{sec:numerics_Lshaped}

We consider problem~\eqref{eq:transmission} posed on the L-shaped domain $(-1/4,1/4)^2\setminus[0,1/4)^2$.
The data $f$, $g$ and $\phi$ are chosen in such a way that the exact solution is given by
\begin{align*}
u(x_1,x_2) &= r^{2/3} \sin(2\varphi/3),\\
\uext(x_1,x_2) &= \frac{1}{2} \log \left[ (x_1+1/8)^2 + (x_2-1/8)^2 \right]
- \frac{1}{2} \log \left[ (x_1-1/8)^2 + (x_2+1/8)^2 \right],
\end{align*}
where the expression of the solution in the interior domain is written
using polar coordinates
(i.e., $r = \sqrt{x_1^2+x_2^2}$ and $\varphi\in(\pi/2,2\pi)$).
In particular, we note that $f=0$,
that the compatibility condition~\eqref{eq:compatibility} is satisfied,
and that the solution exhibits a singularity at the reentrant corner $(0,0)$.

Starting from an initial mesh $\TT_0$ made of 12 elements (see Figure~\ref{fig:initial_mesh_lshaped}),
we run Algorithm~\ref{alg:adaptive_algorithm} with $\theta = 1$ (uniform mesh refinement) and $\theta = 1/4$
(adaptive mesh refinement).

\begin{figure}[ht]
\centering
\begin{subfigure}[b]{0.31\textwidth}
\centering
\includegraphics[trim = 40 30 45 30, clip, width=\textwidth]{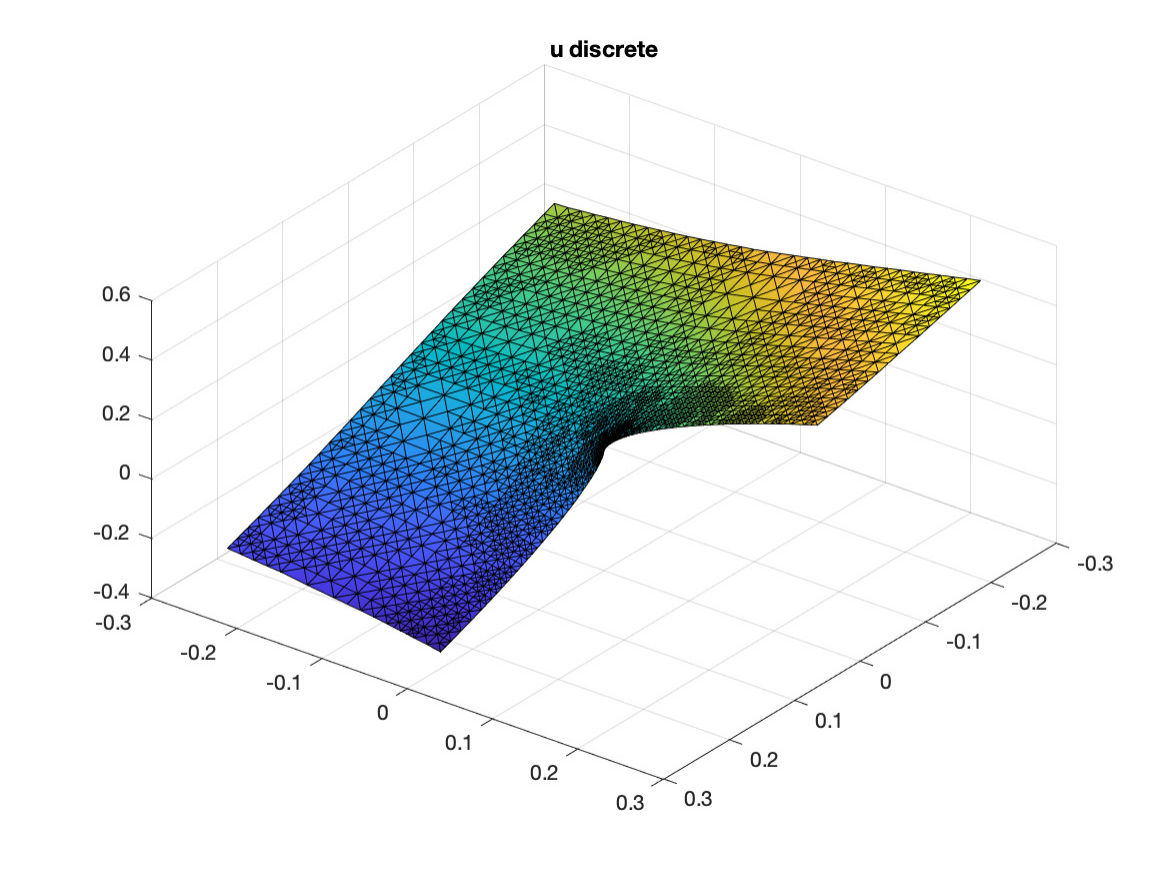}
\caption{$u_{14} = u_{1,14} + u_{2,14}$.}
\end{subfigure}
\hfill
\begin{subfigure}[b]{0.31\textwidth}
\centering
\includegraphics[trim = 40 30 45 30, clip, width=\textwidth]{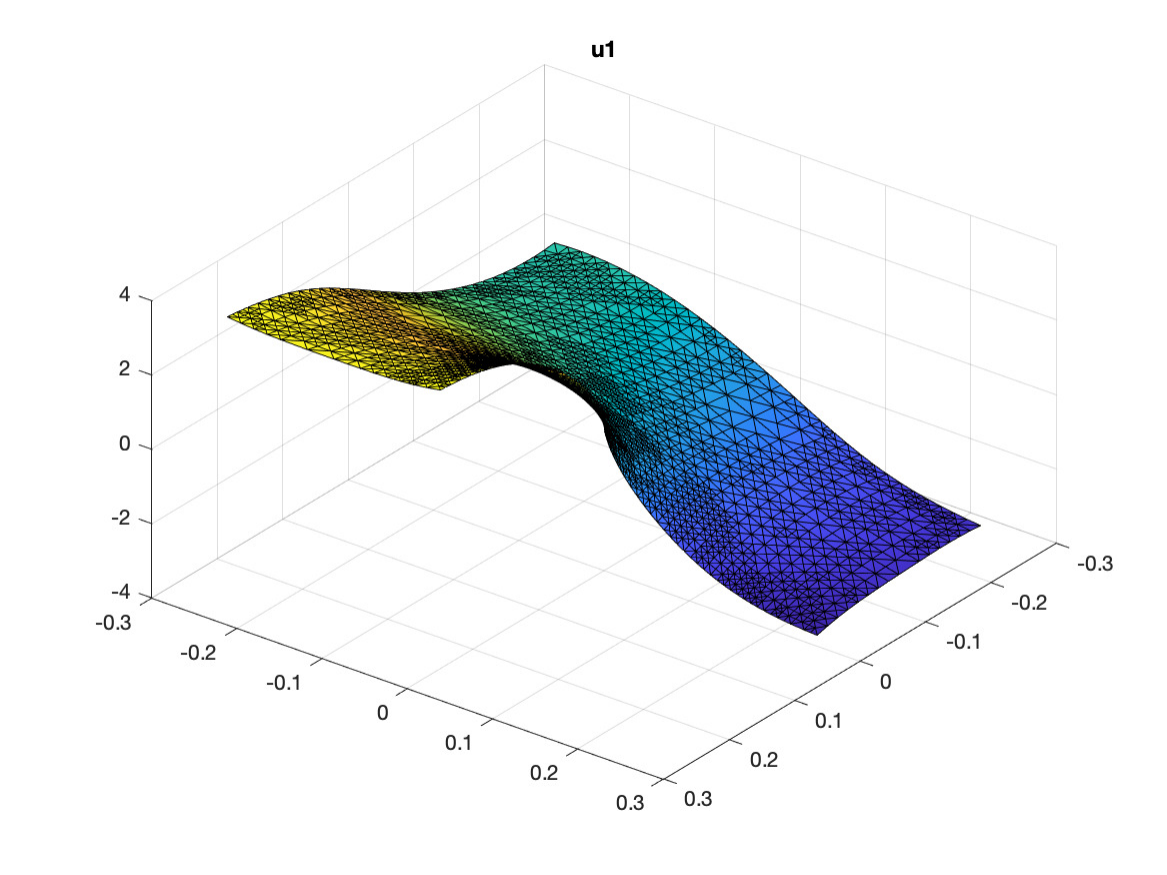}
\caption{$u_{1,14}$.}
\end{subfigure}
\hfill
\begin{subfigure}[b]{0.31\textwidth}
\centering
\includegraphics[trim = 40 30 45 30, clip, width=\textwidth]{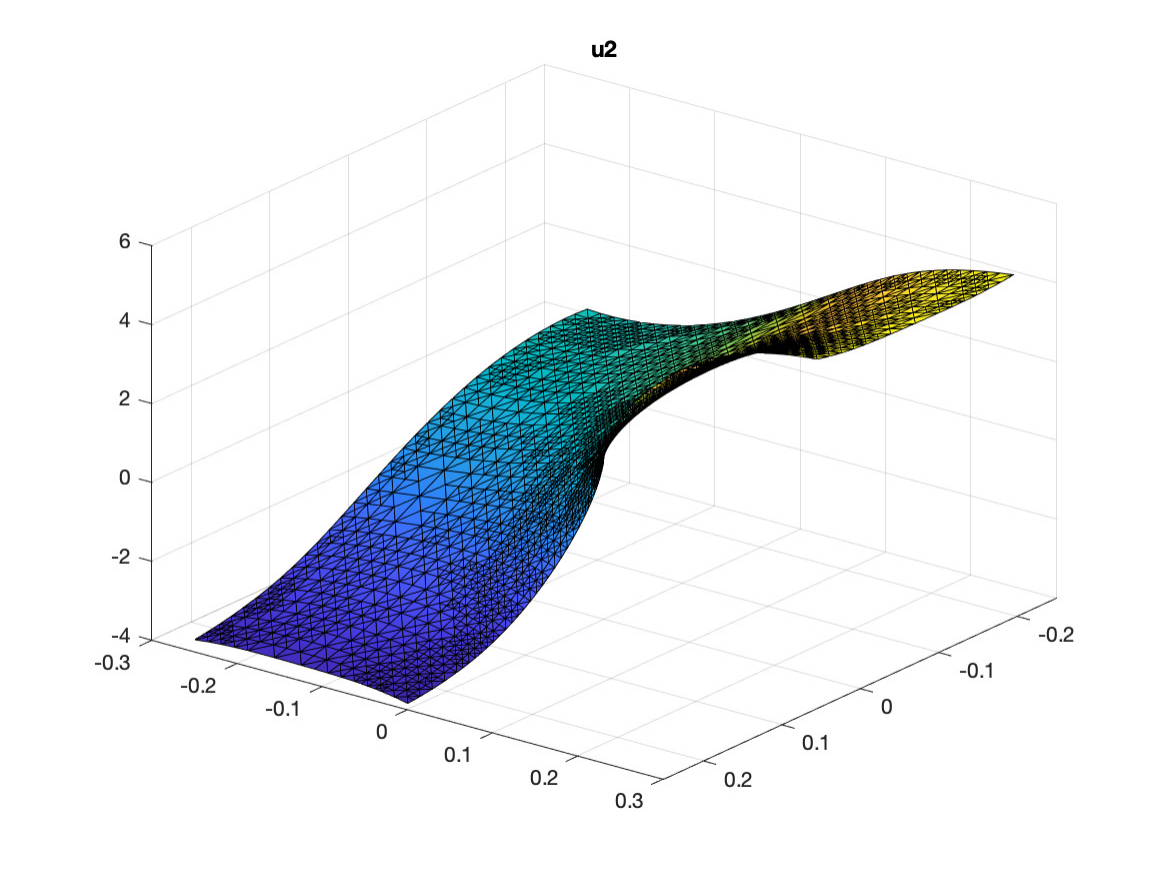}
\caption{$u_{2,14}$.}
\end{subfigure}
\caption{Experiments in Section~\ref{sec:numerics_Lshaped}:
Plots of the approximation $u_{14}$ and of its components $u_{1,{14}}$ and $u_{2,{14}}$.}
\label{fig:Lshaped_approx}
\end{figure}

In Figure~\ref{fig:Lshaped_approx},
we show the plot of the adaptive approximation $u_{14}$
together with those of its components $u_{1,14}$ and $u_{2,14}$
in the decomposition $u_{14} = u_{1,14} + u_{2,14}$.
We see that all approximations exhibit a singularity at the reentrant corner.
From the plot of $u_{14}$, we can clearly see the underlying adaptively refined mesh, which is, as expected, strongly refined where the singularity occurs.

\begin{figure}[ht]
\centering
\begin{subfigure}[b]{0.48\textwidth}
\centering
\begin{tikzpicture}
\pgfplotstableread{pics/Lshaped/results_Lshaped_uniform.dat}{\squareuniform}
\begin{loglogaxis}[
width = \textwidth,
xmin = 7,
xmax = 2.4e6,
ymin=2e-4,
ymax=15,
xlabel={$N_\ell$},
ymajorgrids=true, xmajorgrids=true, grid style=dashed, 
tick label style={font=\tiny},
legend style={
legend pos=south west,
fill=none, draw=none,
font=\tiny
}
]
\addplot[red, mark=diamond, thick] table[x=dofs,y=error]{\squareuniform};
\addplot[blue, mark=+, thick] table[x=dofs,y=est]{\squareuniform};
\addplot[orange, mark=o, thick, densely dash dot] table[x=dofs,y=est1]{\squareuniform};
\addplot[cyan, mark=x, thick, densely dash dot] table[x=dofs,y=est2]{\squareuniform};
\addplot[black,dashed,domain=10^(0):10^(7)] {0.5*x^(-1/3) };
\node at (axis cs:1e1,1.0e-1) [anchor=north west] {\tiny $\mathcal{O}(N_\ell^{-1/3})$};
\legend{
$\norm{u - u_\ell}{H^1(\Omega)}$,
$\eta_\ell$,
$\eta_{1,\ell}$,
$\eta_{2,\ell}$,
}
\end{loglogaxis}
\end{tikzpicture}
\caption{Algorithm~\ref{alg:adaptive_algorithm} with $\theta=1$.}
\label{fig:Lshaped_convergence_uniform}
\end{subfigure}
\hfill
\begin{subfigure}[b]{0.48\textwidth}
\centering
\begin{tikzpicture}
\pgfplotstableread{pics/Lshaped/results_Lshaped_adaptive.dat}{\squareadaptive}
\begin{loglogaxis}[
width = \textwidth,
xmin = 7,
xmax = 2.4e6,
ymin=2e-4,
ymax=15,
xlabel={$N_\ell$},
ymajorgrids=true, xmajorgrids=true, grid style=dashed, 
tick label style={font=\tiny},
legend style={
legend pos=south west,
fill=none, draw=none,
font=\tiny
}
]
\addplot[red, mark=diamond, thick] table[x=dofs,y=error]{\squareadaptive};
\addplot[blue, mark=+, thick] table[x=dofs,y=est]{\squareadaptive};
\addplot[orange, mark=o, thick, densely dash dot] table[x=dofs,y=est1]{\squareadaptive};
\addplot[cyan, mark=x, thick, densely dash dot] table[x=dofs,y=est2]{\squareadaptive};
\addplot[black,dashed,domain=10^(0):10^(7)] {10*x^(-1/2) };
\node at (axis cs:7e4,1e-2) [anchor=north west] {\tiny $\mathcal{O}(N_\ell^{-1/2})$};
\legend{
$\norm{u - u_\ell}{H^1(\Omega)}$,
$\eta_\ell$,
$\eta_{1,\ell}$,
$\eta_{2,\ell}$,
}
\end{loglogaxis}
\end{tikzpicture}
\caption{Algorithm~\ref{alg:adaptive_algorithm} with $\theta=1/4$.}
\label{fig:Lshaped_convergence_adaptive}
\end{subfigure}
\caption{
Experiments in Section~\ref{sec:numerics_Lshaped}:
Error $\norm{u - u_\ell}{H^1(\Omega)}$
and error estimates $\eta_\ell$, $\eta_{1,\ell}$, and $\eta_{2,\ell}$
plotted against the total number of degrees of freedom $N_\ell$.
Comparison of uniform and adaptive mesh refinement.
}
\label{fig:Lshaped_convergence}
\end{figure}

In Figure~\ref{fig:Lshaped_convergence},
for both uniform and adaptive mesh refinements,
we plot the error $\norm{u - u_\ell}{H^1(\Omega)}$
and the error estimates $\eta_\ell$, $\eta_{1,\ell}$, and $\eta_{2,\ell}$
against the total number of degrees of freedom $N_\ell$.
For both approaches,
the error is overestimated by the total error estimate
and
the total error estimate decays with the same rate as the error.
However,
the method with adaptive mesh refinement achieves the optimal rate $\mathcal{O}(N_\ell^{-1/2})$
observed in the previous section in the case of a smooth solution,
whereas for the method with uniform mesh refinement
we observe the suboptimal rate $\mathcal{O}(N_\ell^{-1/3})$.

\subsection{Z-shaped domain} \label{sec:numerics_Zshaped}

We conclude the section by considering problem~\eqref{eq:transmission}
posed on the Z-shaped domain $(-1/4,1/4)^2\setminus\mathrm{conv}\{ (0,0) , (1/4,0) , (1/4,1/4) \}$.
We choose the data
$f = 1$, $g = 0$ and $\phi = -7/[8(8+\sqrt{2})]$,
for which an explicit expression of the exact solution is not available.
However, note that they do satisfy
the compatibility condition~\eqref{eq:compatibility}.
Starting from an initial mesh $\TT_0$ made of 14 elements (see Figure~\ref{fig:initial_mesh_zshaped}),
again we run Algorithm~\ref{alg:adaptive_algorithm} with $\theta = 1$ and $\theta = 1/4$.

\begin{figure}[ht]
\centering
\begin{subfigure}[b]{0.31\textwidth}
\centering
\includegraphics[trim = 30 30 45 30, clip, width=\textwidth]{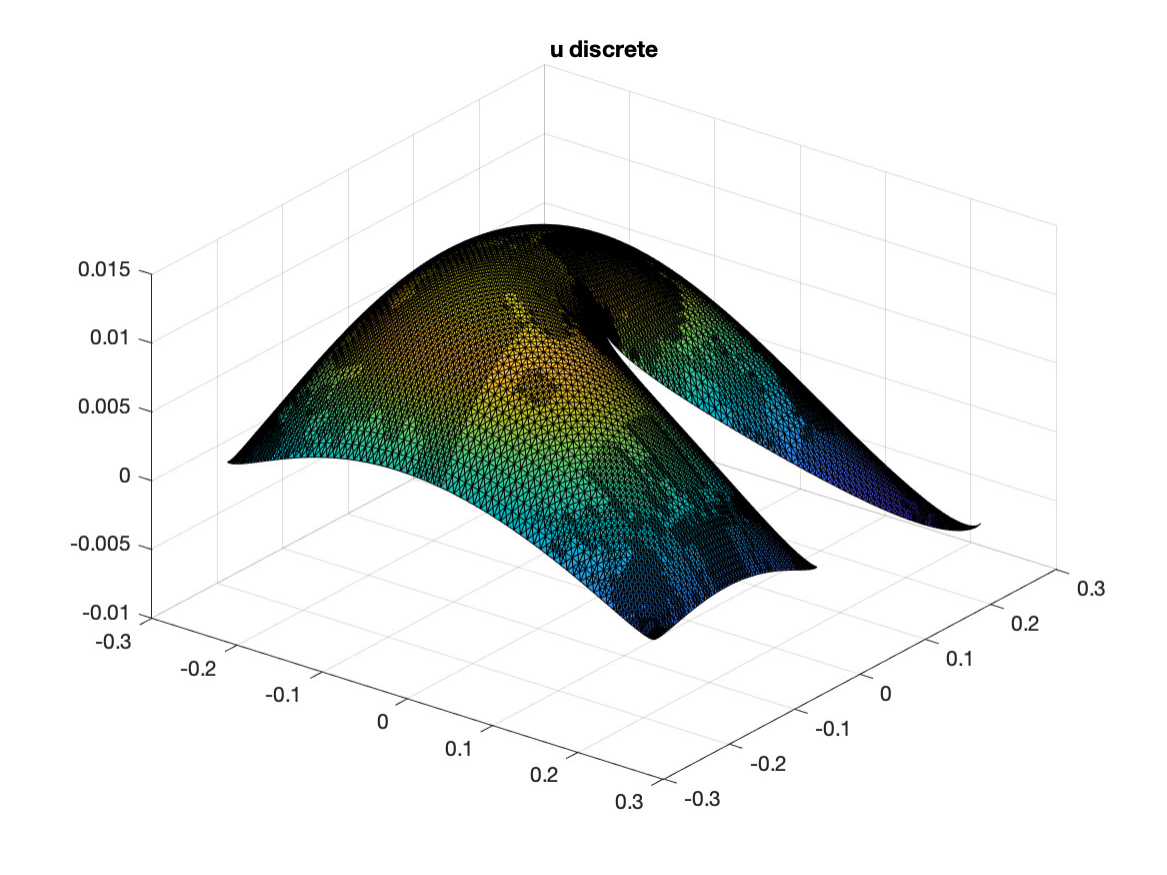}
\caption{$u_{10} = u_{1,10} + u_{2,10}$.}
\end{subfigure}
\hfill
\begin{subfigure}[b]{0.31\textwidth}
\centering
\includegraphics[trim = 30 30 45 30, clip, width=\textwidth]{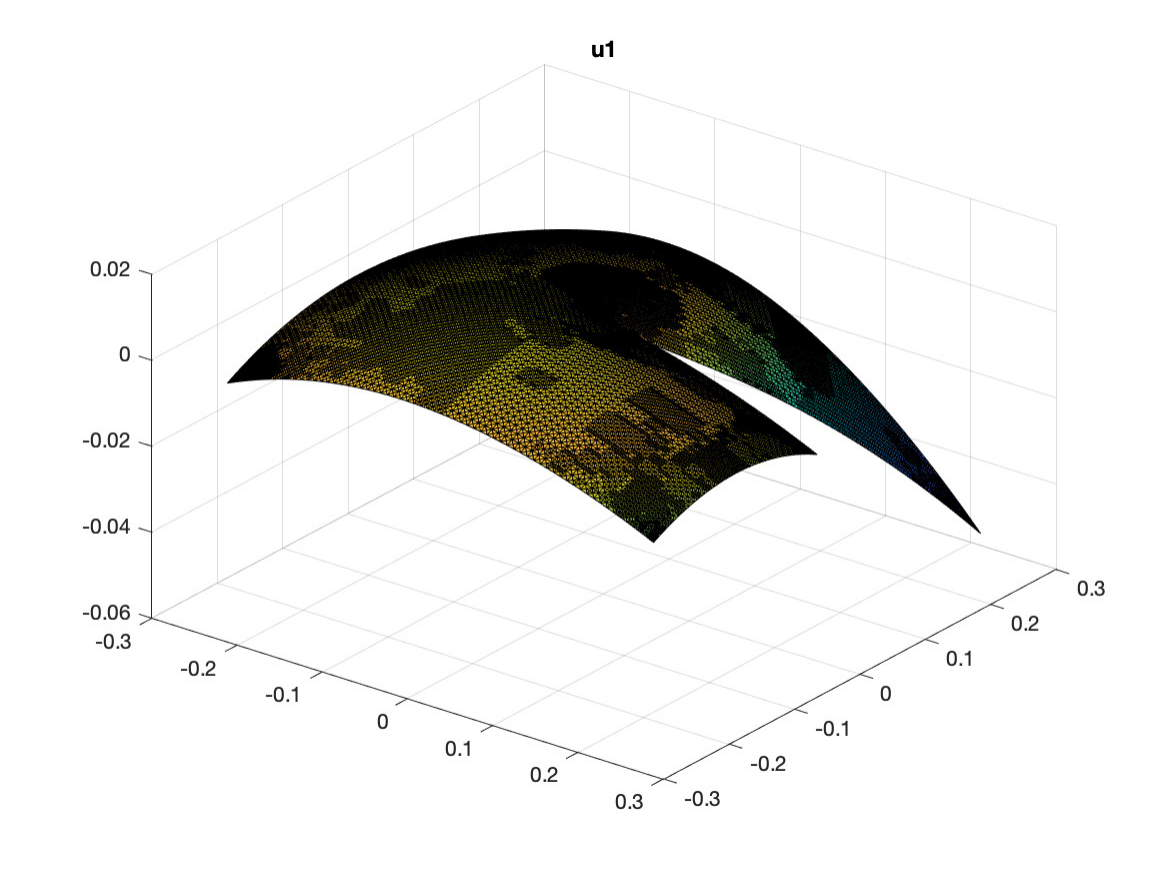}
\caption{$u_{1,10}$.}
\end{subfigure}
\hfill
\begin{subfigure}[b]{0.31\textwidth}
\centering
\includegraphics[trim = 30 30 45 30, clip, width=\textwidth]{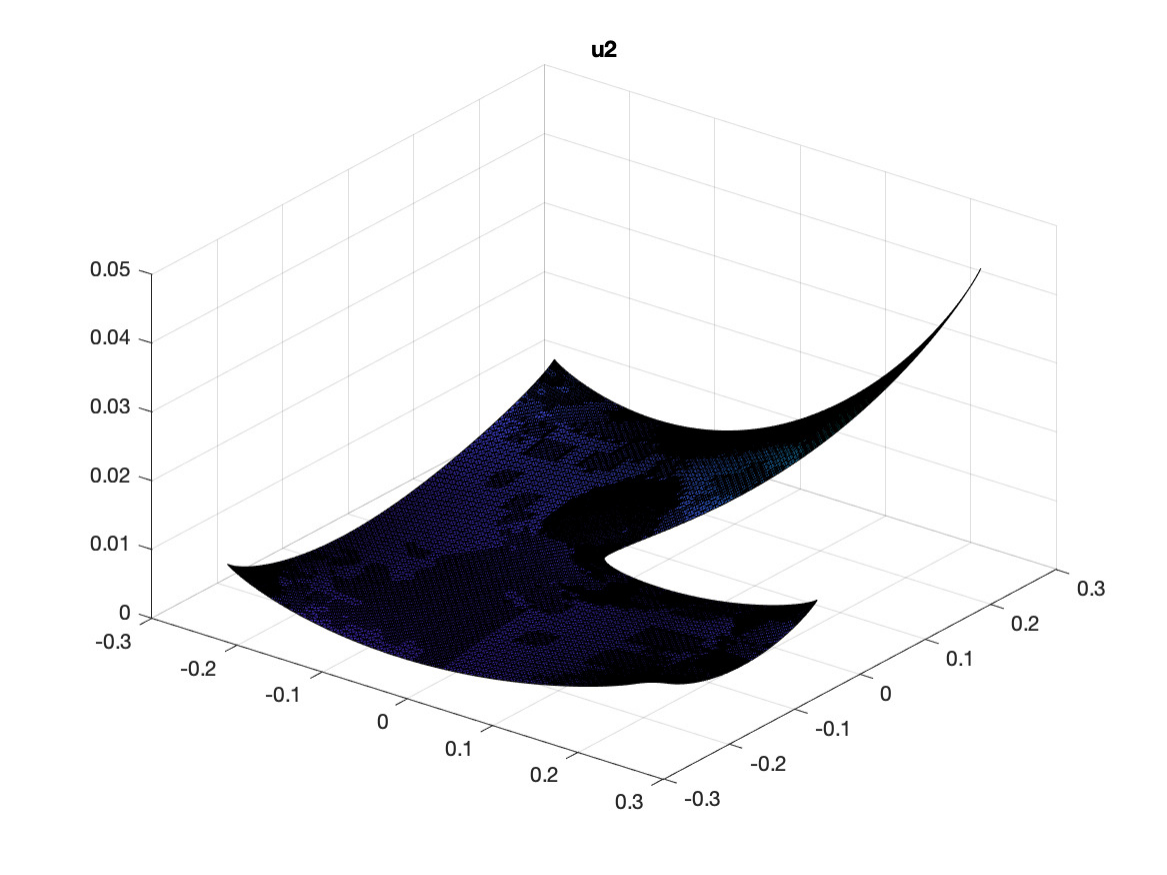}
\caption{$u_{2,10}$.}
\end{subfigure}
\caption{Experiments in Section~\ref{sec:numerics_Zshaped}:
Plots of the approximation $u_{10}$ and of its components $u_{1,{10}}$ and $u_{2,{10}}$.}
\label{fig:Zshaped_approx}
\end{figure}

The adaptive approximation $u_{10}$ and its components $u_{1,{10}}$ and $u_{2,{10}}$
are depicted in Figure~\ref{fig:Zshaped_approx}.
We observe that the adaptively refined mesh identifies the geometric singularity at
the reentrant corner.

\begin{figure}[ht]
\centering
\begin{subfigure}[b]{0.48\textwidth}
\centering
\begin{tikzpicture}
\pgfplotstableread{pics/Zshaped/results_Zshaped_uniform.dat}{\squareuniform}
\begin{loglogaxis}[
width = \textwidth,
xmin = 7,
xmax = 2.4e6,
ymin=1.2e-4,
ymax=1.2e-1,
xlabel={$N_\ell$},
ymajorgrids=true, xmajorgrids=true, grid style=dashed, 
tick label style={font=\tiny},
legend style={
legend pos=south west,
fill=none, draw=none,
font=\tiny
}
]
\addplot[blue, mark=+, thick] table[x=dofs,y=est]{\squareuniform};
\addplot[orange, mark=o, thick, densely dash dot] table[x=dofs,y=est1]{\squareuniform};
\addplot[cyan, mark=x, thick, densely dash dot] table[x=dofs,y=est2]{\squareuniform};
\addplot[black,dashed,domain=10^(0):10^(7)] {0.04*x^(-5/14) };
\node at (axis cs:1e2,1.0e-3) [anchor=south west] {\tiny $\mathcal{O}(N_\ell^{-5/14})$};
\legend{
$\eta_\ell$,
$\eta_{1,\ell}$,
$\eta_{2,\ell}$,
}
\end{loglogaxis}
\end{tikzpicture}
\caption{Algorithm~\ref{alg:adaptive_algorithm} with $\theta=1$.}
\label{fig:Zshaped_convergence_uniform}
\end{subfigure}
\begin{subfigure}[b]{0.48\textwidth}
\centering
\begin{tikzpicture}
\pgfplotstableread{pics/Zshaped/results_Zshaped_adaptive.dat}{\squareuniform}
\begin{loglogaxis}[
width = \textwidth,
xmin = 7,
xmax = 2.4e6,
ymin=1.2e-4,
ymax=1.2e-1,
xlabel={$N_\ell$},
ymajorgrids=true, xmajorgrids=true, grid style=dashed, 
tick label style={font=\tiny},
legend style={
legend pos=south west,
fill=none, draw=none,
font=\tiny
}
]
\addplot[blue, mark=+, thick] table[x=dofs,y=est]{\squareuniform};
\addplot[orange, mark=o, thick, densely dash dot] table[x=dofs,y=est1]{\squareuniform};
\addplot[cyan, mark=x, thick, densely dash dot] table[x=dofs,y=est2]{\squareuniform};
\addplot[black,dashed,domain=10^(0):10^(7)] {1*x^(-1/2) };
\node at (axis cs:1e5,8e-3) [anchor=north west] {\tiny $\mathcal{O}(N_\ell^{-1/2})$};
\legend{
$\eta_\ell$,
$\eta_{1,\ell}$,
$\eta_{2,\ell}$,
}
\end{loglogaxis}
\end{tikzpicture}
\caption{Algorithm~\ref{alg:adaptive_algorithm} with $\theta=1/4$.}
\label{fig:Zshaped_convergence_adaptive}
\end{subfigure}
\caption{
Experiments in Section~\ref{sec:numerics_Zshaped}:
Error estimates $\eta_\ell$, $\eta_{1,\ell}$, and $\eta_{2,\ell}$
plotted against the total number of degrees of freedom $N_\ell$.
Comparison of uniform and adaptive mesh refinement.
}
\label{fig:Zshaped_convergence}
\end{figure}
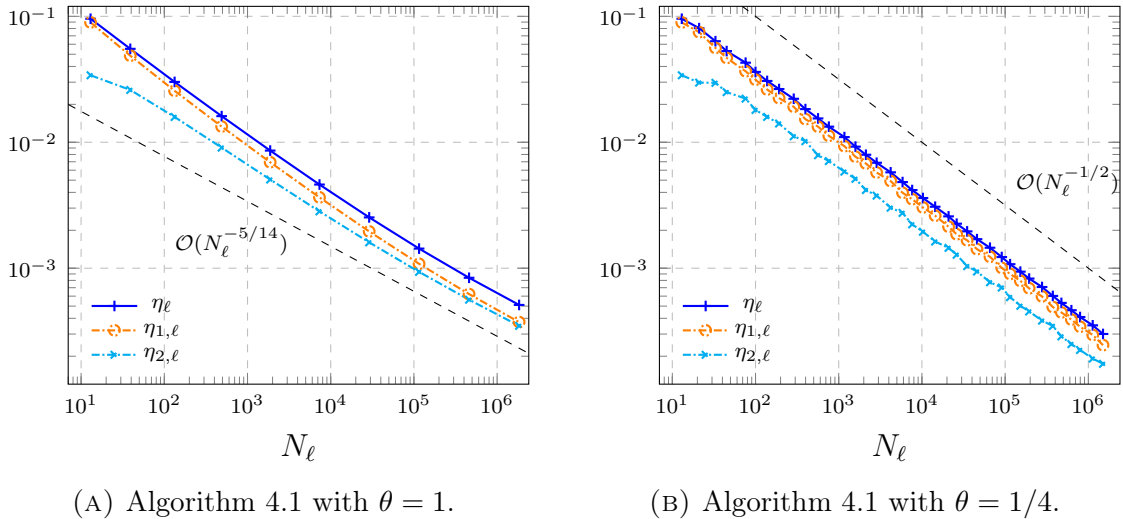

In Figure~\ref{fig:Zshaped_convergence},
we show the decay of the error estimates $\eta_\ell$, $\eta_{1,\ell}$, and $\eta_{2,\ell}$
as the total number of degrees of freedom $N_\ell$ increases.
The convergence rate is optimal for adaptive mesh refinement
and suboptimal for uniform mesh refinement.

\bibliographystyle{alpha}
\bibliography{literature}
\end{document}